\documentclass[11pt]{amsart}
\usepackage{amsfonts,amssymb,amsxtra}

\usepackage{eufrak}
\usepackage{latexsym} 

\newcommand{\el}{\par \mbox{} \par \vspace{-0.5\baselineskip}}
\newcommand{\goth}[1]{\EuFrak{#1}}

\newcounter{amoi}
\newtheorem{theo}{Theorem}
\newtheorem{leme}{Lemma}
\newtheorem{prop}{Proposition}
\newtheorem{coro}{Corollary}
\newtheorem{defi}{Definition}
\newcommand{\noi}{\noindent}
\newenvironment{rem}{\noi {\el \noi \bf Remark}}{ \el }

\begin{document}
\begin{center}
  {\Large {Cohomology of generalized supergrassmannians and character formulae for basic classical Lie superalgebras} \vskip 2cm}

Caroline {\sc Gruson} {\footnote{U.M.R. 7502 du CNRS, Institut Elie Cartan, 
Universite Henri Poincar\'e (Nancy 1), BP 239,
54506 Vandoeuvre-les-Nancy Cedex, France. E-mail:
Caroline.Gruson@iecn.u-nancy.fr}}  and Vera {\sc Serganova} {\footnote
{Department of Mathematics, University of California, Berkeley, CA,94720-3840 USA.  E-mail:
      serganov@Math.Berkeley.EDU}}

\end{center}

\bigskip

\section{Introduction}
The famous Borel-Weil-Bott theorem states that, for a reductive 
complex algebraic group $G$, the cohomology groups of any invertible sheaf on 
the flag variety $G/B$
are non zero in at most one degree, and in this degree the cohomology
group is a simple $G$-module. This statement is no longer true for
Chevalley groups in characteristic $p$, the cohomology group is
non-zero only in one degree (for a dominant weight) but it is not a simple $G$-module
anymore. Its structure is rather complicated and interesting,  and
studying it leads to Kazhdan-Lusztig theory.

Let now $G$ be a complex basic classical supergroup. Studying the cohomology
groups on the flag supervariety $G/B$ was initiated by Penkov and
others (\cite{MPV}, \cite{P}).
They discovered that the Borel-Weil-Bott theorem  is not always
true in this case and proved that it is true for so called typical
invertible sheaves. 

This question is closely related to the representation theory 
of the corresponding basic classical Lie superalgebra $\goth g$. The
category of finite-dimensional $\goth g$-modules is not
semi-simple, only typical simple modules do not have extensions with
other simple modules. The problem of computing characters of atypical simple
finite dimensional $\goth g$-modules is a full-time occupation for
several persons since 1977. The formula for typical highest weight
modules was found in \cite{Krep}. In 1980 Bernstein and Leites
found the character formula in case $\goth g=\goth{sl}(1,n)$
(\cite{BL}), later
this formula was generalised  for $\goth {osp}(2,2n)$
and for the so called singly atypical modules and generic atypical
modules (see \cite{VdJ} and \cite{JHKT}). Similar type formulae were
conjectured in \cite{KW} for affine superalgebras.
In 1996, the second author \cite{V.Sel} solved the problem 
completely for $\goth{gl}(m,n)$ using the geometry of flag
supervarieties. Later Brundan \cite{B} solved this problem by a
different, purely algebraic, method
and discovered a remarkable connection with the representation theory of
$\goth{gl}(\infty)$. Brundan's method was developed further in \cite{CW},
\cite{BS1}, \cite{BS2}, \cite{BS3}, \cite{CL}.

In this paper we generalise the method of \cite{V.Sel} to the orthosymplectic
groups $OSP(m,2n)$. The results were announced in \cite{V.ICM} without proof
and one statement (typical lemma) was formulated with a mistake there. We
take a different Borel subgroup in this paper and use the language of
weight diagrams (invented by Brundan and Stroppel for $GL(m,n)$). This makes
the combinatorics much easier and allows one to formulate a
precise algorithm for the computation of the characters of all simple modules.
In the $\mathfrak {gl}$ case, there exist  finite dimensional analogues of 
Verma modules, called Kac modules, which can be interpreted as cohomology 
groups of bundles over the flag supermanifold, and which also have a nice 
algebraic behavior. Unfortunately (?), there is no such analogue in the 
orthosymplectic case, so we use the geometrical method.

Let us explain the main idea of the method. As was pointed out by
Penkov, one can easily write down the character of the Euler
characteristic for any invertible sheaf on a flag supervariety using classical
Borel-Weil-Bott theory. Thus, if one knows the multiplicities of simple
modules in cohomology groups, one hopes to express the characters of
simple modules as linear combinations of characters of Euler
characteristic of some invertible sheaves. We manage to find these
multiplicities for some generalized super grassmannians $G/P$ for certain
parabolic subgroups $P$, and use this information to compute characters.

There are several complications in the case of $OSP(m,2n)$  compared 
with $GL(m,n)$. First, in many cases the Euler characteristic is zero
on $G/B$, and one should consider instead $G/Q$ for a suitable choice
of a parabolic subgroup $Q$ which depends on the module whose character
we want to calculate. Next, there is no distinguished Borel subgroup
in $OSP(m,2n)$, hence there are non-vanishing cohomology groups in many
degrees. Finally, the recursion process involving the computation of the
cohomology
is more complicated due to the existence of the so-called
exceptional pairs (see Section 10 in the paper for details). On the
other hand, there is one advantage: the character of a simple
$OSP(m,2n)$ -module is always a finite linear combination of Euler
characteristics, in contrast with the $GL(m,n)$ case where the combination
is always infinite.

The organisation of the paper is as follows. In Section 2 we fix the
notations. Section 3 contains general statements about cohomology
groups of vector bundles on a supergrassmannian $G/P$. In Section 4 we
compute the character of the Euler characteristic, and, for parabolic subgroups
$P\subset Q\subset G$, relate the cohomology
on $G/P$ with those on $G/Q$ and $Q/P$, using Leray spectral sequence. 
Section 5 contains a complete classification of the blocks in the
category of finite-dimensional $\goth g$-modules (Theorem ~\ref{bl})
and introduces translation functors.
In Section 6, we introduce weight diagrams and translate the results
of Section 5 into this language. In Section 7
we explain how to reduce the computation of the cohomology to the
case of the most atypical blocks of $\mathfrak{gl}(k,k)$,
$\goth{osp}(2k+1,2k),\goth{osp}(2k,2k)$ and $\goth{osp}(2k+2,2k)$. (Afterwards 
we concentrate on the orthosymplectic case, since the general linear case 
is done in \cite{V.Sel}).
This last computation is done recursively in Sections 8,9 and 10 starting with
$\goth{osp}(2k+1,2k)$ (resp. $\goth{osp}(2k,2k)$,
$\goth{osp}(2k+2,2k)$) and getting down to $\goth{osp}(2k-1,2k-2)$
(resp. $\goth{osp}(2k-2,2k-2)$,
$\goth{osp}(2k,2k-2)$). In Section 11, the recursion is solved. We construct a 
graph whose vertices are weight diagrams and describe a
combinatorial algorithm for the computation of cohomology groups, in terms of
paths in this graph (Proposition ~\ref{legmovprop}
and Theorem ~\ref{combth}). Section 12 presents an algorithm for
calculating characters (Theorem ~\ref{thch}) and contains some examples.
In Section 13, we present a simplfication of the Section 11's algorithm. In the
general linear case, this simplification leads to the equivalence (proven
in \cite{MS}) of
algorithms appearing in \cite{V.Sel} and \cite{B}. 

Both authors thank Laurent Gruson for hospitality and stimulating
discussions. The second author would like to thank Ian Musson for
explaining the method of weight diagrams and fruitful discussions on supergeometry.

\section{Notations} 

Let $\goth g = \goth g _0 \oplus \goth g _1$ be a basic classical
complex Lie superalgebra, i.e. $\goth g=\goth {gl}(m,n)$ or
$\goth{osp}(m,n)$. By $G$ we denote the linear algebraic groups $GL(m,n)$ and 
$OSP(m,n)$ 
respectively ~\cite{M}. 
For a Lie subalgebra (e.g. $\goth a\subset \goth g$), we denote with the
capital Latin letter (in this case $A\subset G$) the connected 
complex algebraic supergroup with the given Lie superalgebra. 
For all the superalgebras considered in this paper such supergroups
are well defined.

We fix a Cartan subalgebra $\goth h$ of $\goth g$.
Denote by $\Delta$ the set of roots of $\goth g$ with respect to $\goth h$.
Let $\goth b$ be a Borel subalgebra of $\goth g$
containing $\goth h$, it defines a set of positive roots $\Delta^+$, 
and we define $\rho=\rho_0-\rho_1$, where $\rho_0$ is 
the half sum of the positive even roots and $\rho_1$ is the half sum of positive odd
roots. Recall that since $\goth g$ is basic classical, it is equipped
with a non degenerate invariant bilinear form and the restriction of
this bilinear form to $\goth h$ is also non degenerate. We will denote
this form by $(\; , \;)$. We denote by $W$ the Weyl group of the even
part $\goth g_0$. Recall that in our case the Cartan subalgebra 
$\goth h$ of $\goth g$ is also a Cartan subalgebra of $\goth g_0$. 
For a Lie subalgebra $\goth a\subset \goth g$ such that $\goth
h\subset \goth a$ we denote by 
$\Delta(\goth a)\subset \Delta$ the set of roots of $\goth a$.

A weight $\lambda\in \goth h^*$ is {\it integral} if it
induces a one-dimensional representation of the Cartan subgroup $H$. In
the paper we consider only integral weights. Thus, by a weight we always mean 
an integral weight. Moreover, we only consider modules which are integrable
with respect to the group $G$: all those have integral weights. 
Define the standard order
on the set of integral weights: $\lambda\leq\mu$ iff
$\mu-\lambda=\sum_{\alpha\in\Delta^+}n_{\alpha}\alpha$ where all
$n_{\alpha}$ are non-negative integers.

For any integral weight $\lambda$, we will denote by $L_{\lambda}$ the 
 simple $\goth g$-module with highest weight $\lambda$, and if $\goth a$ is a 
Lie subalgebra of $\goth g$ for which it makes sense,
we will denote by $L_{\lambda}(\goth a)$ the irreducible $\goth
 a$-module with highest weight $\lambda$. Recall that $\lambda$ is
 called {\it dominant} (resp. $\goth a$-dominant) if $L_{\lambda}$
 (resp. $L_{\lambda}(\goth a)$) is finite-dimensional (in this
 case $L_{\lambda}$ (resp. $L_{\lambda}(\goth a)$) has a natural
 structure of $G$-module (resp. $A$-module)).

Let $\mathcal U (\goth g)$ be the universal enveloping algebra of
$\goth g$ and $\mathcal Z(\goth g)$ be its center. 
For every weight $\lambda$, we write $\chi _{\lambda}$ for the corresponding central character.
A central character $\chi$ is dominant if there exists a dominant
$\lambda$ such that $\chi=\chi _{\lambda}$.

Finally, let us recall the description of $\Delta$ (see ~\cite{Kadv}). Let $\goth g=\goth
{gl}(m,n), \goth {osp}(2m,2n)$ or $ \goth {osp}(2m+1,2n)$. Then
$\operatorname{dim} \goth h=m+n$ and one can choose a basis
$\varepsilon_1,...,\varepsilon_m,\delta_1,...,\delta_n$ of $\goth h^*$
such that
$$ (\varepsilon_i,\varepsilon_j)=\delta_{ij},\,\,(\varepsilon_i,\delta_j)=0,\,\,(\delta_i,\delta_j)=-\delta_{ij}.$$

The even roots $\Delta_0$ of $\goth {gl}(m,n)$ are all vectors of the
form $\varepsilon_i-\varepsilon_j$ and $\delta_i-\delta_j$ with $i\neq
j$. The odd roots $\Delta_1$ of $\goth {gl}(m,n)$ are all vectors of
the form  $\varepsilon_i-\delta_j$ and $\delta_i-\varepsilon_j$.

The even roots $\Delta_0$ of $\goth {osp}(2m,2n)$ are all vectors of the
form $\pm\varepsilon_i\pm\varepsilon_j$, $\pm\delta_i\pm\delta_j$
(the signs can be chosen independently) with $i\neq j$ and $2\delta_i$. The odd roots $\Delta_1$ of 
$\goth {osp}(2m,2n)$ are all vectors of
the form  $\pm \varepsilon_i\pm \delta_j$. 

The even roots $\Delta_0$ of $\goth {osp}(2m+1,2n)$ are all vectors of the
form $\pm\varepsilon_i\pm\varepsilon_j$, $\pm\delta_i\pm\delta_j$
with $i\neq j$, $\pm\varepsilon_i$ and $\pm 2\delta_i$. The odd roots $\Delta_1$ of 
$\goth {osp}(2m+1,2n)$ are all vectors of
the form  $\pm \varepsilon_i \pm\delta_j$ and $\pm\delta_i$.

We define a parity on the weight lattice by saying that $\varepsilon _i$
(resp. $\delta _j$) is even (resp. odd). Then the category of finite 
dimensional $G$-modules splits in a direct summand of two parts, one in which 
the weight spaces have the same parity as the corresponding weight, and one
in which the parities differ. In this paper, we will only consider the
first part.

\section{Geometric induction}

Let $\goth p$ be any parabolic subalgebra of $\goth g$
containing 
$\goth b$ and $\goth l$ denote the reductive part of $\goth p$. 

For a $P$-module e.g. $V$, we denote by the calligraphic
letter $\mathcal V$ the vector bundle $G\times _P V$ over the
generalized grassmannian $G/P$. Note that the space of sections of $\mathcal V$ on
any open set has a natural structure of $\goth g$-module, in other
words the sheaf of sections of $\mathcal V$ is a $\goth g$-sheaf. 
Therefore the cohomology groups $(H^i(G/P,\mathcal V))$ are $\goth
g$-modules. For details see \cite{M}, \cite{P}, \cite{MPV}. 

Define the functor $\Gamma_i$ from the category of $\goth p$-modules to the
category of $\goth g$-modules by
$$\Gamma _i (G/P,  V) := (H^i(G/P,\mathcal V ^*))^*.$$

\begin{leme}\label{indgen} - The functors $\Gamma _i$ have the following properties: 

1) if
$$0 \rightarrow U \rightarrow V \rightarrow W \rightarrow 0$$
is a short exact sequence of $P$-modules, then one has
$$\ldots \rightarrow  \Gamma _1 (G/P,  W) \rightarrow \Gamma _0 (G/P,  U) \rightarrow \Gamma _0 (G/P,  V) \rightarrow \Gamma _0 (G/P, W) \rightarrow 0$$
(long exact sequence).

2) if $M$ is a $\goth g$-module and $V$ is a $P$-module, the following holds :
$$\Gamma _i (G/P , V \otimes M) = \Gamma _i (G/P,V)\otimes M.$$
\end{leme}

The proof is an adaptation of standard arguments (see \cite{J}).

\begin{leme}\label{l2} - The module $\Gamma_0(G/P,V)$ is the maximal
finite-dimensional quotient of $\mathcal U(\goth g)\otimes_{\mathcal U(\goth  p)}V$.
\end{leme}
\begin{proof} - Let $\tilde{V}$ be the maximal
  finite-dimensional quotient of 
$\mathcal U(\goth g)\otimes_{\mathcal U(\goth  p)}V$.
By duality $\tilde{V}^*$ is the maximal finite-dimensional submodule in
the coinduced module
$\operatorname{Hom}_{\mathcal U(\goth p)}(\mathcal U(\goth g),V^*)$.

Let $\Gamma (V^*):=H^0(G/P,\mathcal V^*)$. By definition
$$\Gamma(V^*)=\left\{\gamma\in\mathbb C[G]\otimes V^* |
  \gamma(gp)=p^{-1}\gamma(g), g\in G, p\in P\right\}.$$

Let $\pi$ denote the composition of the standard maps
$V\to \mathcal U(\goth g)\otimes_{\mathcal U(\goth  p)}V\to\tilde{V}$
and $\pi^*:\tilde{V}^*\to V^*$ be the dual map. Then
$\gamma_v(g)=\pi^*(g^{-1}v)$ is a vector in $\Gamma (V^*)$. Hence we
can define a $G$-module homomorphism $\varphi: \tilde{V}^*\to \Gamma
(V^*)$ by putting $\varphi(v):=\gamma_v$. We claim that $\varphi$ is
injective. Indeed, assume $X=\operatorname{Ker}\varphi\neq 0$. But
$\pi^*(X)\neq 0$ for every
non-trivial $G$-submodule $X\subset \tilde{V}^*$. Therefore $X$ contains 
$v$ such that $\gamma_v(e)\neq 0$, hence $\varphi(X)\neq 0$. Contradiction.

On the other hand, the map $ev: \Gamma (V^*)\to V^*$
defined by $ev(\gamma):=\gamma(e)$ induces a homomorphism of $\goth
g$-modules 
$j:\Gamma (V^*)\to \operatorname{Hom}_{\mathcal U(\goth p)}(\mathcal U(\goth g),V^*)$. 
Since $\Gamma (V^*)$ is finite-dimensional $j(\Gamma (V^*))\subset
\tilde {V}^*$. Again we claim that $j$ is injective. 
Indeed, choose a nilpotent superalgebra $\goth m^-$
such that $\goth g=\goth p\oplus \goth m^-$, and let $M^-$ be the
corresponding supergroup. It is known that (exactly as in classical
case) $M^-P$ is dense in $G$ (see \cite{MV}). Hence 
$ev$ is injective on the subspace of invariants $\Gamma(V^*)^{\goth  m^-}$. 
Let $Y=\operatorname{Ker}j$. Then $Y^{m^-}=0$. That implies $Y=0$. 

Thus, we have two injective morphisms $\varphi: \tilde{V}^*\to \Gamma
(V^*)$ and $j:\Gamma (V^*)\to \tilde {V}^*$. Since the spaces are
finite-dimensional, both $\varphi$ and $j$ are isomorphisms. Therefore,
the dual modules $\tilde{V}$ and $\Gamma _0 (G/P,  V)$ are also isomorphic. 

\end{proof}

\begin{defi}\label{penkovrem} (Penkov's remark \cite{P}) - 
  Let $(X,\mathcal{O}_X)$ be a generalized
  grassmannian, and $(X_0,\mathcal O_{X_0})$ be the underlying
  algebraic variety, $\mathcal{V}$ be a $G$-vector bundle on $X$. Then the
  corresponding sheaf $\tilde {\mathcal V}$ on $X_0$ has a filtration by
  $\mathcal{O}_{X_0}$-modules such that the 
associated graded module  (in degree $i$) is isomorphic to the $G_0$-bundle 
$\mathcal {V}_{X_0} \otimes  S^i(\mathcal{N}^*_{X_0}X)=:\tilde{\mathcal{V}}^i$.
The whole module will be denoted by $\mathcal I _{X_0}(\mathcal {V} \otimes  S^{\bullet}(\mathcal{N}^*_{X_0}X))$.

By definition $H^k(X, \mathcal V)=H^k(X_0, \tilde {\mathcal V})$. We have, if we denote by $Ch(V)$ the character of a $\goth g$ module $V$,
\[
Ch (H^k(X_0, Gr (\tilde {\mathcal V})))\geq Ch(H^k(X,\mathcal{V}))
\]
as can be seen with the long exact sequence coming with the filtration
of $\tilde {\mathcal V}$. The $\geq$ means that each $\goth
g_0$-module occurring in $H^k(X, {\mathcal V})$ appears
in $H^k(X_0, Gr (\tilde {\mathcal V}))$ with at least the same multiplicity.
Note that the inequality becomes an equality when computing the Euler 
characteristic.
\end{defi}

\begin{leme} \label{l1} - If $L_{\mu}$ occurs in $\Gamma_i(G/P, 
{L}_{\lambda}(\goth p))$ with non-zero multiplicity, then 
$\mu +\rho  =w(\lambda +\rho)-\sum_{\alpha \in I}\alpha$ for some $w\in W$ of length $i$ and  $I \subset \Delta_1^+$. 
\end{leme}
\begin{proof} - We make use of Penkov's remark. First, note that 
$$Gr(S^i(\mathcal N _{X_0} X))=G_0 \times_{P_0} S^i(\goth g /(\goth g _0 \oplus \goth p _1 ))$$
which is a $G_0$-bundle on $G_0/P_0$, so: 
$$Ch(H^k(G/P, \mathcal L_{\lambda}(\goth p)^*)) \leq Ch(H^k(G_0 /P_0,\mathcal I _{G_0/P_0}\mathcal  L_{\lambda}^*(\goth p )\otimes S^{\bullet}(\goth g / (\goth g _0 \oplus \goth p _1 ))^*)).$$
Assume $L_{\mu}$ occurs in $\Gamma _i (G/P, L_{\lambda} ^*
(\goth p))$. It means that $L_{\mu} ^*$ occurs in $H^i (G/P,
\mathcal I _{G_0/P_0}\mathcal L_{\lambda} ^* (\goth p ))$. 
Therefore $L_{\mu}(\goth g_0)^*$ occurs in $H^i(G_0/P_0, \mathcal
L_{\lambda}^*(\goth p)\otimes S^{\bullet}(\goth g / (\goth g _0 \oplus
\goth p _1 ))^*)$. The latter sheaf has a filtration by $G_0$-bundles
whose simple quotients are of the form $\mathcal L_{\lambda
  -\sum_{\gamma \in J \subset \Delta^+_1}\gamma}^*(\goth p_0)$. By the usual (non-graded)
Borel-Weil-Bott theorem, if $w(\nu+\rho_0)-\rho_0$ is regular and
$\goth g_0$-dominant, then
$$H^i(G_0/P_0, \mathcal L_{\nu}(\goth p _0)^*) = \left\{
  \begin{array}{c}L_{w(\nu +\rho_0)-\rho_0}(\goth g_0)^* \; if \; l(w)= i \\ 0 \; otherwise \end{array}\right .$$
Therefore
$$\mu=w(\lambda -\sum_{\gamma \in J \subset
  \Delta_1^+}\gamma+\rho_0)-\rho_0.$$

Using $\rho_0=\rho+\rho_1$, we obtain
$$\mu=w(\lambda+\rho)-\rho-\sum_{\alpha \in I \subset
  \Delta_1^+}\alpha,$$
where $I=(w(J)\cap \Delta_1^+)\cup -(w(\Delta_1^+-J)\cap \Delta_1^-)$. 
\end{proof}

Write $\goth p=\goth l \oplus \goth m$, $\goth g=\goth m^-\oplus\goth
l \oplus \goth m$, where $\goth m$ is the nilpotent part of $\goth p$. 
Consider the projection 
$$\phi:\mathcal U(\goth g)=\mathcal U(\goth m^-)\mathcal U(\goth l) \mathcal U(\goth m) \to
\mathcal U(\goth m^-)\mathcal U(\goth l)$$ 
with the kernel $\mathcal U(\goth g)\goth m$.
The restriction of $\phi$ to $ \mathcal Z(\goth g)$ induces the injective homomorphism of centers 
$ \mathcal Z(\goth g)\to  \mathcal Z(\goth l)$. 
We denote  by $\Phi$ the dual map 
$$Hom (\mathcal Z(\goth l),\mathbb{C})\to Hom ( \mathcal Z(\goth g),\mathbb {C}).$$

\begin{leme}\label{l4} - If $V$ is an irreducible $\goth p$-module admitting a
  central character $\chi$, then the $\goth g$-module
  $\Gamma_i(G/P, V)$ admits the central character $\Phi (\chi)$.
\end{leme}
\begin{proof} - Consider another projection
$$\psi:\mathcal U(\goth g)=\mathcal U(\goth m)\otimes \mathcal U(\goth l)\otimes 
\mathcal U(\goth m^-) \to\mathcal U(\goth l)\mathcal U(\goth m^-)$$ 
with the kernel $\goth m \mathcal U(\goth g)$, and the corresponding
dual map
$$\Psi:Hom (\mathcal Z(\goth l),\mathbb{C})\to Hom ( \mathcal Z(\goth g),\mathbb {C}).$$

Note that any element $x\in \mathcal U(\goth g)$ acts on the set of
sections of the vector bundle $\mathcal V^*$ over any open subset $U\subset G/P$. We
denote this action by $L_x$.

We claim that if $\gamma$ is a section of $\mathcal V^*$ 
over any open subset $U\subset G/P$ and $z\in \mathcal Z(\goth g)$ belongs to the
kernel of $\psi$, then $L_z(\gamma)=0$. Since $z$ commutes with the group
action, it suffices to prove that $L_z(\gamma)\in I_P$, where $I_P$
stands for the subspace of sections which are zero at $P$. Our claim
now follows immediately from the facts that, if 
$x\in \goth m$, then $L_x(\gamma)\in I_P$, and if $x\in \goth l$, then
the value of $L_x(\gamma)$ at $p\in P$ equals $x(\gamma(p))$. Thus, if $\zeta$ is the central character of
 $\goth l$-module $V^*$, then the space of sections $\Gamma(U,\mathcal V^*)$ (as a $\goth g$-module) admits central character $\Psi(\zeta)$.

Using the Chech complex for the computation of the cohomology groups of $\mathcal V^*$, one can show easily that the 
cohomology groups $H^i(G/P,\mathcal V^*)$ also admit central character $\Psi(\zeta)$.
Going to the dual modules provides the statement.
\end{proof}

The following corollary of Lemma ~\ref{indgen}(2) and Lemma ~\ref{l4}
will be used a lot in this paper.

\begin{coro}\label{coroblocks} - For any finite-dimensional
  $\goth g$-module $M$ (resp. finite-dimensional $\goth p$-module $V$) let
  $M^{\chi}$ (resp. $V^{\Phi^{-1}(\chi)}$)
denote the component with generalized central character $\chi$
  (resp. with generalized central character lying in
  $\Phi^{-1}(\chi)$). Then
$$\Gamma_i(G/P,(V\otimes M)^{\Phi^{-1}(\chi)})=(\Gamma_i(G/P,V)\otimes M)^{\chi}.$$
\end{coro}

Let $\mathcal F$ be the category of finite dimensional $\goth
g$-modules semisimple over $\goth h$; this category decomposes into blocks $\mathcal F ^{\chi}$, where $\mathcal F ^{\chi}$ consists of all finite dimensional modules with (generalized) central character $\chi$.

\begin{rem} - Let $\goth l$ be the Levi subalgebra of $\goth p$. 
  If $V$ is a module belonging to the block
  $\mathcal F ^{\tau}(\goth l)$ consisting all finite dimensional modules with (generalized) central character $\tau$,
  then $\Gamma _i (G/P,  V)$ belongs to
  the block $\mathcal F ^{\Phi(\tau)}$.
  That provides a correspondence
  between blocks of $\goth l$ and blocks of $\goth g$.  
\end{rem}

\begin{defi} \label{DAl}- Let $\lambda$ be a $\goth g$-dominant weight. Define
  $A(\lambda)$ to be a maximal possible set of mutually orthogonal
  positive isotropic roots $\alpha _i$ of $\goth g$ such that $(\lambda + \rho , \alpha _i) =0$, $A(\lambda) = \{\alpha _1 ,\ldots \alpha _l \}$. We put
$\# A(\lambda) = \# \lambda$, and call it the degree of atypicality of
  $\lambda$ (say it is $0$ if $A(\lambda) = \emptyset$: then $\lambda$
  is called typical). 
\end{defi}

Although the choice of $A(\lambda)$ is not unique, the degree of
atypicality does not depend on it.

Then, for any weight $\mu$, $\chi _{\lambda} = \chi _{\mu}$ is
equivalent to the fact that $\mu$ can be written as $w(\lambda + \rho
+ n_1 \alpha _1 + \ldots + n_l \alpha _l) - \rho$, where $w\in W$ and
$n_i \in \mathbb C$ for all $i = 1, \ldots, l$ (see ~\cite{V.adv}).

Notice that if  $\chi _{\lambda} = \chi _{\mu}$, then $\lambda$ and
$\mu$ have the same degree of atypicality. So the degree of
atypicality is a well defined notion for a central character.

For any non-isotropic $\beta \in \Delta$ put $\check {\beta}:=\frac{2\beta}{(\beta,\beta)}$.

\begin{defi}\label{Dpa} - A parabolic subalgebra $\goth p \subset \goth g$ with Levi
  part $\goth l$ is  called admissible for a central character $\chi$
if, for any dominant $\lambda$ such that 
  $\chi_{\lambda}=\chi$, one has
  $({\lambda}+\rho,\check{\beta}) \geq 0$ for all 
$\beta \in\Delta_0^+ - \Delta(\goth l)$.
\end{defi}

For example, if $\goth g=\goth {gl}(m,n)$, then 
a distinguished Borel subalgebra (\cite{Kadv}) is admissible for any
central character, and therefore so is any parabolic subalgebra containing
this distinguished Borel subalgebra. If the simple roots of a Borel subalgebra are all
isotropic, then $\rho=0$ and any parabolic subalgebra containing this
Borel subalgebra is admissible for any central character.

\begin{leme}\label{typ} ({\bf Typical lemma}) - Let $\lambda$ be a dominant
  weight, let $\goth p$ be a parabolic subalgebra of
  $\goth g$ admissible for $\chi_{\lambda}$.
  Assume $A(\lambda)\subset \Delta(\goth l)$ and $({\lambda}+\rho,\check{\beta}) > 0$ for all 
$\beta \in\Delta_0^+ - \Delta(\goth l)$ 
  (in this case we will call $\lambda$  $\goth p$-typical). 
Then,
$$\Gamma _i(G/P,  L_{\lambda}(\goth p)) = \left \{ \begin{array}{c} 0 \; if \; i>0 \\ L_{\lambda}(\goth g) \; if \; i=0 \end{array}\right .$$ 
\end{leme}

\begin{proof} - Let $L_{\mu}$ be an irreducible subquotient in
  $\Gamma _i(G/P, L_{\lambda}(\goth p))$. Then by 
  Lemma \ref{l1} there exist $w\in W$ and $I\subset\Delta^+_1$ such
  that
\begin{equation}\label{loc1}
\mu=w(\lambda+\rho)-\rho-\sum_{\alpha \in  I}\alpha.
\end{equation} 
Choose an element $h\in\goth h^*$ such that
$$\Delta(\goth p)=\{\alpha\in\Delta | (\alpha,h)\geq 0 \}.$$
Note that
$\Delta(\goth l)=\{\alpha\in\Delta | (\alpha,h)= 0\}$.

We claim that $(\mu,h)\leq (\lambda,h)$. Indeed, let $s\in W$ be such
that $s(\lambda+\rho)$ belongs to the positive Weyl chamber. Since $(\lambda+\rho,\check{\beta}) \geq 0$ for all 
$\beta \in\Delta_0^+ - \Delta(\goth l)$,  $s$ belongs to the Weyl group
  of $\goth l_0$. Since $s(h)=h$ we have 
$(s(\lambda+\rho),h)=(\lambda+\rho,h)$. Then
$$w(\lambda+\rho)=s(\lambda+\rho)-\sum_{\beta\in\Delta_0^+}k_{\beta}\beta.$$
  with some non-negative $k_{\beta}$ since $s(\lambda+\rho) $ lies in the positive Weyl chamber. Therefore
  $$(w(\lambda+\rho),h)=(s(\lambda+\rho),h)-\sum_{\beta\in\Delta_0^+}k_{\beta}(\beta,h)\leq (s(\lambda+\rho),h)=(\lambda+\rho,h).$$ 
Therefore by (\ref{loc1})
$$(\mu+\rho,h)=(w(\lambda+\rho),h)-\sum_{\alpha \in  I}(\alpha,h)\leq
(w(\lambda+\rho),h)\leq (\lambda+\rho,h).$$

On the other hand, $\chi_{\lambda}=\chi_{\mu}$ implies
$$u(\mu+\rho)=\lambda+\rho+\sum_{\alpha\in A(\lambda)}k_{\alpha}\alpha$$
for some $u\in W$. Since 
$\mu$ is dominant, we obtain by the same argument as above

$$(\mu+\rho,h)\geq (u(\mu+\rho),h)=(\lambda+\rho,h)+\sum_{\alpha\in
  A(\lambda)}k_{\alpha}(\alpha,h)=(\lambda+\rho,h).$$

Hence $(\mu,h)\geq(\lambda,h)$. Thus, in fact $(\lambda,h)=(\mu,h)$.

That implies that $(w(\lambda+\rho),h)=(\lambda+\rho,h)$ and
$I\subset\Delta^+_1(\goth l)$. 
Since $({\lambda}+\rho,\check{\beta}) > 0$ for all 
$\beta \in\Delta_0^+ - \Delta(\goth l)$,  $w$
belongs to the Weyl group of $\goth l_0$. 
It follows from the construction of 
$w$ in the proof of Lemma ~\ref{l1} that 
$w=\operatorname{id}$ and $i=0$. 
By Lemma \ref{l2}, $L_{\mu}$ is a subquotient of 
$\mathcal U(\goth g)\otimes_{\mathcal U(\goth p)} L_{\lambda}(\goth p)$. Therefore  ${\mu}={\lambda}$.
\end{proof}

Note that the proof of the above lemma implies the following

\begin{coro} \label{firstmark} - Let $\lambda$ be a dominant weight, 
$\goth p$ be admissible for $\chi_{\lambda}$ and $h\in\goth h^*$ be such that
$\Delta(\goth p)=\{\alpha\in\Delta | (\alpha,h)\geq 0 \}.$ If
$L_{\mu}$ occurs in $\Gamma _i(G/P,  L_{\lambda}(\goth p))$ with
non-zero multiplicity, then $(\mu,h)\leq(\lambda,h)$. If, in addition, 
$({\lambda}+\rho,\check{\beta}) > 0$ for all 
$\beta \in\Delta_0^+ - \Delta(\goth l)$, then  $(\mu,h)=(\lambda,h)$
implies $i=0$ and $\mu=\lambda$.
\end{coro}

\section{Induction for geometric induction}

 We choose two parabolic subalgebras of $\goth g$
containing $\goth b$ such that $\goth b \subset \goth p \subset \goth
q \subset \goth g$.

The aim of the present section is to relate, for a $\goth p$-dominant weight $\mu$, the cohomology 
$\Gamma _{\bullet} (G/Q ,L_{\lambda} (\goth q))$ with  both $\Gamma
_{\bullet} (G/P , L_{\lambda} (\goth p))$ and $\Gamma _{\bullet} (Q/P
, L_{\mu} (\goth p))$.

\begin{defi}\label{Dpoin} - For $A,B$ any $P,Q,G$ such that $B \subset A$, we define the Poincar\'e polynomial in the variable $z$: 
$$K_{A,B} ^{\lambda , \mu} (z) := \sum _{i \geq 0} [\Gamma _{i} (A/B , L_{\lambda} (\goth b)) : L_{\mu}(\goth a)] z^i.$$
We denote by $^i K _{A,B} ^{\lambda , \mu}$ the coefficient of $z^i$.

\end{defi}

The following result was first stated in \cite{V.ICM}:
\begin{prop}\label{char} (Euler caracteristic formula) - 
Let $\goth p$ be a parabolic subalgebra of $\goth g$, denote by $\goth l$ 
its Levi part and set
$$D_0 = \Pi _{\alpha \in \Delta _0 ^+}(e^{\alpha / 2}- e^{-\alpha / 2}),
D_1 = \Pi _{\alpha \in \Delta _1 ^+}(e^{\alpha / 2}+ e^{-\alpha / 2}),
D=\frac{D_1}{D_0}.$$
For any $\goth p$-dominant weight $\lambda$, one has :
\begin{equation}\label{char1}
\sum _{\mu}K_{G,P}^{\lambda , \mu}(-1) Ch (L_{\mu}) = D 
\sum _{w\in W}\varepsilon (w) w(\frac{e^{\rho}Ch(L_{\lambda}(\goth p))}
{\Pi _{\alpha \in \Delta_1 ^+(\goth l)}(1+e^{-\alpha})}).
\end{equation}
\end{prop}

\begin{proof} - The left  hand side of the formula is
$$\sum _{i =0} ^{dim G_0/P_0} (-1)^i Ch (H^i(G/P, \mathcal L _{\lambda}(\goth p)^*).$$
Using Definition \ref{penkovrem}, one gets that the left hand side becomes
$$\sum_{i=0}^{dimG_0/P_0} (-1)^i Ch(H^i(G_0/P_0, (\oplus _k \tilde {\mathcal L} _{\lambda} (\goth p)^k)^*)^*).$$
We use the classical Borel-Weil-Bott theory to get
$$\sum_{i=0}^{dimG_0/P_0} (-1)^i Ch(H^i(G_0/P_0, (\oplus _k \tilde
{\mathcal L} _{\lambda} (\goth p)^k)^*)^*) = \frac{1}{D_0}
\sum_{w\in W} \varepsilon (w)w(Ch(\oplus _k \tilde {\mathcal L}_{\lambda} (\goth p)^k)e^{\rho _0}).$$
Remember now that $\oplus _k \tilde {\mathcal L} _{\lambda} (\goth p)^k = \mathcal I _{G_0/P_0}(\mathcal L _{\lambda}(\goth p)\otimes S^{\bullet} (\goth g / (\goth g_0 \oplus \goth p _1)))$. A direct computation gives the proposition.

\end{proof}
\begin{theo}\label{ind} - One has :
$$K_{G,P} ^{\lambda , \mu} (-1) = \sum _{\nu} K_{Q,P} ^{\lambda , \nu} (-1)  K_{G,Q} ^{\nu , \mu} (-1) .$$
\end{theo}

\begin{proof} - Denote by $\pi : G/P \longrightarrow G/Q$ the canonical
  projection. The fibre of $\pi$ is isomorphic to $Q/P$. Consider the
  derived functor (of sheaves) $R^{\bullet} \pi _*$, it transforms 
$\mathcal L _{\lambda }(\goth p)^*$ into a complex of sheaves $R^{\bullet}\pi _* (\mathcal L _{\lambda }(\goth p)^*)$ over $G/Q$. 

Take an injective resolution $\mathcal {L}_k$ of $\mathcal
L _{\lambda }(\goth p)^*$,
 over $G/P$ and then  an injective resolution of
$\pi_*(\mathcal L _k)$ over $G/Q$.  This gives a bicomplex of sheaves
over $G/Q$, and its cohomology is isomorphic to
$H^{\bullet}(G/P,\mathcal L _{\lambda }(\goth p)^*)$. 

On the other hand, the Leray spectral sequence of this bicomplex has the term
$$E_2^{p,q}=H^p(G/Q,R^q\pi_*(\mathcal L _{\lambda }(\goth p)^*)).$$
One has (\cite{J})
$$R^q\pi_*\mathcal L _{\lambda }(\goth p)^*=H^q(Q/P,\mathcal L
_{\lambda }(\goth p)^*_{| fibre}).$$
By definition of the coefficient $^qK _{A,B} ^{\lambda , \mu}$,
  we have the following identities in the Grothendieck groups:
$$[H^q(Q/P,\mathcal L
_{\lambda }(\goth p)^*_{| fibre})]=\sum_{\nu} {}^qK _{Q,P} ^{\lambda ,
    \nu} [L_\nu(\goth q)^*],$$
$$[E_2^{p,q}]=\sum_{\mu}\sum_{\nu}{} ^pK _{G,Q}^{\nu,\mu}{}^qK
_{Q,P}^{\lambda,\nu}[L_{\mu}^*].$$
The theorem follows when one computes the Euler characteristic.
\end{proof}

\section{Blocks}

Recall that we let  $\mathcal U (\goth g)$ be the universal enveloping
algebra of $\goth g$ and $\mathcal Z(\goth g)$ be its center. For every weight 
$\lambda$, we write $\chi _{\lambda}$ for the corresponding central character.

The aim of the present section is to prove the following:

\begin{theo}\label{bl} 
- Let $\lambda$ be a dominant weight with atypicality degree $k$, then
the block $\mathcal F ^{\chi _{\lambda}}$ is equivalent to the maximal
atypical block of $\goth g _k$
containing the trivial module, where $\goth g _k$ is the following:

if $\goth g = \goth {gl}(m,n)$ then $\goth g _k = \goth {gl}(k,k)$

if  $\goth g = \goth {osp}(2m+1,2n)$ then $\goth g _k = \goth {osp}(2k+1,2k)$

if  $\goth g = \goth {osp}(2m,2n)$ then $\goth g _k = \goth {osp}(2k,2k)$
or $\goth {osp}(2k+2,2k)$
\end{theo}

In what follows the Borel subalgebra $\goth b$ is such that every simple
root of $\goth b_0$ is either a simple root of $\goth b$ or a sum of two
odd simple roots of $\goth b$.

\begin{leme}\label{dom} - A weight $\lambda$ is dominant,  i.e. $L_{\lambda}$ is  
finite-dimensional if and only if, for any simple root $\alpha$ of $\goth b_0$, 

$\bullet$ $(\lambda+\rho,\check\alpha)\in \mathbb Z_{>0}$ if $\alpha$
or $\frac{\alpha}{2}$ is simple in $\goth b$;

$\bullet$ if $\alpha=\gamma+\beta$ is a sum of two isotropic simple
roots then $(\lambda+\rho,\check\alpha)\in \mathbb Z_{>0}$ or
$(\lambda+\rho,\beta)=(\lambda+\rho,\gamma)=0$;

$\bullet$ Finally if $\alpha=\beta+\gamma$ where $\gamma$ is an odd
isotropic simple root and $\beta$ is an odd non-isotropic simple root,
then $(\lambda+\rho,\check\alpha)\in \mathbb Z_{>0}$ or
$(\lambda+\rho,\check\alpha)=-1$ and $(\lambda+\rho,\gamma)=0$.
\end{leme}

\begin{proof} - For an arbitrary Borel subalgebra $\goth b'$ containing
  $\goth b_0$ let $\lambda(\goth b')$ denote the highest weight of
  $L_{\lambda}$ with respect to $\goth b'$ and $\rho(\goth b')$ be the
  analogue of $\rho$ for $\goth b'$. A weight $\lambda$ is
  dominant (see ~\cite{V.KM}) if for any simple root
  $\alpha$ of $\goth b_0$ there exists a Borel subalgebra $\goth b'$
  such that $\alpha$ or $\frac{\alpha}{2}$ is simple in $\goth b'$ and
  $(\lambda(\goth b')+\rho(\goth b'),\check\alpha)\in \mathbb Z_{>0}$.

If $\alpha$ or $\frac{\alpha}{2}$ is a simple root of $\goth b$, 
then we are in
  the situation of the first two cases and the statement is straightforward.
If $\alpha=\beta +\gamma$, then choose $\goth b'$ obtained by odd
  reflection with respect to $\gamma$. Then the statement follows from
  the following formulae
$$\lambda(\goth b')+\rho(\goth b')=\lambda+\rho{ } \text { if } (\lambda+\rho,\gamma)\neq 0,$$
$$\lambda(\goth b')+\rho(\goth b')=\lambda+\rho+\gamma { } \text { if } (\lambda+\rho,\gamma)=0.$$
\end{proof}

\begin{rem}\label{rem1} - 
Note that our choice of $\goth b$ implies that for any simple root
$\alpha$ of $\goth b_0$, $(\rho,\check {\alpha})=1,0,-1$ if $\alpha$
is a simple root of $\goth b$, a sum of two isotropic simple roots or
a sum of one isotropic and one non-isotropic odd simple roots
respectively. The latter case is possible only for
$\goth{osp}(2m+1,2n)$. In particular, in the cases of $\goth {gl}(m,n)$ and
$\goth {osp}(2m,2n)$, every parabolic subalgebra of $\goth g$ containing
$\goth b$ is admissible for all central characters.
\end{rem}

The above conditions on a Borel subalgebra 
determine $\goth b$ uniquely up to an automorphism of 
$\goth g$ if $\goth g=\goth {osp}(m,2n)$. In the case of 
$\goth g=\goth{gl}(m,n)$ we choose the distinguished $\goth b$. 

Here we list the simple roots for our choice of Borel subalgebra:

$\bullet$ If $\goth g=\goth {gl}(m,n)$, $m\geq n$, the simple roots are
$$\varepsilon_1-\varepsilon_2,\varepsilon_2-\varepsilon_3,...,\varepsilon_m-\delta_1,\delta_1-\delta_{2},...,\delta_{n-1}-\delta_n,$$
$$\rho=\frac{m-n-2}{2}\varepsilon_1+\frac{m-n-4}{2}\varepsilon_2+...+\frac{-m-n}{2}\varepsilon_m+
\frac{m+n}{2}\delta_1+...+\frac{m-n+2}{2}\delta_n;$$

$\bullet$ If $\goth g=\goth {osp}(2m+1,2n)$ and $m\geq n$, the simple
roots are
$$\varepsilon_1-\varepsilon_2,...,\varepsilon_{m-n+1}-\delta_1,\delta_1-\varepsilon_{m-n+2},...,\varepsilon_{m}-\delta_n,\delta_n,$$
$$\rho=(m-n-\frac{1}{2})\varepsilon_1+(m-n-\frac{3}{2})\varepsilon_2+...+\frac{1}{2}\varepsilon_{m-n}-\frac{1}{2}\varepsilon_{m-n-1}-...-\frac{1}{2}\varepsilon_m+\frac{1}{2}\delta_1+...+\frac{1}{2}\delta_n;$$
 
$\bullet$ If $\goth g=\goth {osp}(2m+1,2n)$ and $m<n$, the simple
roots are
$$\delta_1-\delta_2,...,\delta_{n-m}-\delta_1,\delta_1-\varepsilon_{n-m+1},...,\varepsilon_{m}-\delta_n,\delta_n,$$
$$\rho=-\frac{1}{2}\varepsilon_1-...-\frac{1}{2}\varepsilon_{m}+(n-m+\frac{1}{2})\delta_{1}+
(n-m-\frac{1}{2})\delta_2+...+\frac{1}{2}\delta_{n-m+1}+...+\frac{1}{2}\delta_n;$$

$\bullet$ If $\goth g=\goth {osp}(2m,2n)$ and $m>n$, the simple
roots are
$$\varepsilon_1-\varepsilon_2,...,\varepsilon_{m-n}-\delta_1,\delta_1-\varepsilon_{m-n+1},...,\delta_n-\varepsilon_m,\delta_n+\varepsilon_m,$$
$$\rho=(m-n)\varepsilon_1+(m-n-1)\varepsilon_2+...+\varepsilon_{m-n+1};$$

$\bullet$ If $\goth g=\goth {osp}(2m,2n)$ and $m\leq n$, the simple
roots are
$$\delta_1-\delta_2,...,\delta_{n-m+1}-\varepsilon_1,\varepsilon_1-\delta_{n-m+2},...,\delta_n-\varepsilon_m,\delta_n+\varepsilon_m,$$
$$\rho=(n-m)\delta_1+(n-m-1)\delta_2+...+\delta_{n-m+1}.$$

\begin{coro}\label{cordom} Let 
$$\lambda+\rho=a_1\varepsilon_1+...+a_m\varepsilon_m+b_1\delta_1+...+b_n\delta_n.$$
Then $\lambda$ is integral iff $a_i,b_j\in\mathbb Z$ for $\goth
g=\goth{gl}(m,n)$ or $\goth {osp}(2m,2n)$,
$a_i,b_j\in\frac{1}{2}+\mathbb Z$ for $\goth g=\goth{osp}(2m+1,2n)$.
Furthermore,
$\lambda$ is dominant iff the following conditions hold

$\bullet$ If $\goth g=\goth {gl}(m,n)$, $$a_1>a_2>...>a_m \; and \; b_1>b_2>...>b_n;$$

$\bullet$ If $\goth g=\goth {osp}(2m+1,2n)$, either 
$$a_1>a_2>...>a_{m}\geq \frac{1}{2} \; , \; b_1>b_2>...>b_{n}\geq\frac{1}{2},$$ or

$$a_1>a_2>...>a_{m-l-1}>a_{m-l}=...=a_m=-\frac{1}{2}$$ and 
$b_1>b_2>...>b_{n-l-1}\geq b_{n-l}=...=b_n=\frac{1}{2}$ for some $0\leq l\leq min(m,n)$;

$\bullet$ If $\goth g=\goth {osp}(2m,2n)$, either 
$$a_1>a_2>...>a_{m-1} > |a_{m}| \; and  \; b_1>b_2>...>b_{n}>0,$$ or
 
$$a_1>a_2>...>a_{m-l-1}\geq a_{m-l}=...=a_m=0$$
and $b_1>b_2>...>b_{n-l-1}>b_{n-l}=...=b_n=0$ for some $0\leq l\leq min(m,n).$
\end{coro}

Let $\chi=\chi_{\nu}$ be a central character with atypicality degree $k$. Choose
a self-commuting element $x=\sum_{\alpha\in A(\nu)}X_{\alpha}$.
Let $C(x)$ denote the centralizer of $x$ in $\goth g$. Then $[x,\goth g]$
is an ideal in $C(x)$ and one can choose a reductive subalgebra 
$\goth g_x\subset C(x)$ such that $C(x)=\goth g_x \oplus [x,\goth g]$
(see \cite{DS}). 
This choice is canonical if the Cartan subalgebra is fixed. Denote by
$\mathcal U(\goth g)^x$ the set of $ad_x$-invariants in $\mathcal U(\goth g)$. One can
prove (\cite{DS}) that $\mathcal U(\goth g)^x=\mathcal U(\goth g_x)\oplus \mathcal U(\goth g_x)[x,\mathcal U(\goth g)]$.
Consider the projection $p:\mathcal U(\goth g)^x \to \mathcal U(\goth g_x)$ with the
kernel $\mathcal U(\goth g_x)[x,\mathcal U(\goth g)]$. The restriction of $p$ to the
center of $\mathcal Z(\goth g)$ defines a homomorphism $\mathcal
Z(\goth g) \to  \mathcal Z(\goth g_x)$. 
Consider the dual map
$$p^*:\operatorname{Hom}(\mathcal Z(\goth g_x),\mathbb C)\to
\operatorname{Hom}(\mathcal Z(\goth g),\mathbb C).$$
It was shown in \cite{DS} that the preimage
$(p^*)^{-1}(\chi)$ consists of one central character $\chi'\in
\operatorname{Hom}(\mathcal Z(\goth g_x),\mathbb C)$ unless $\goth
g=\goth{osp}(2m,2n)$. If $\goth g=\goth{osp}(2m,2n)$ then
$(p^*)^{-1}(\chi)$ might consist of two  central characters
$\chi'$ and $\chi''$ such that one is obtained from another by an
involutive outer automorphism of $\goth g_x\simeq\goth{osp}(2m-2k,2n-2k)$ 
(induced by an automorphism of $\goth o(2m-2k)$).
In the latter case by $\chi'$ we denote the central character
corresponding to a dominant weight with non-negative marks.

\el

\noi {\bf Terminology} - We call $\chi'$ {\it the core} of $\chi$. 

\el

If $\goth h_x=\goth h\cap \goth g_x$,
$\rho'=\frac{1}{2}(\sum_{\alpha\in\Delta_0^+(\goth g_x)}\alpha-\sum_{\alpha\in\Delta_1^+(\goth g_x)}\alpha)$,
then $\chi'=\chi_{\mu}$, where $\mu+\rho'$ is
the restriction of $\nu+\rho$ to $\goth h_x$.
It is a simple but very important observation that the central character
$\chi$ is uniquely determined by its core $\chi'$.
 
Lemma \ref{dom} and Corollary \ref{cordom} imply the following
 
\begin{leme}\label{Lcore} - The core $\chi'$ is a typical dominant central character of
  $\goth g_x$.
\end{leme} 

Assume that $\chi$ has degree of atypicality $k>0$. Then,
independently of the choice of $\lambda$ such that
$\chi_{\lambda}=\chi$, the Lie superalgebra $\goth g_x$ is isomorphic
to one of the following (see \cite{DS})

$\bullet$ if $\goth g=\goth {gl}(m,n)$ then $\goth g_x \simeq \goth {gl}(p,q)$ with  $p=m-k,q=n-k$; 

$\bullet$ if $\goth g=\goth{osp}(2m+1,2n)$ then $\goth g_x \simeq \goth{osp}(2p+1,2q)$ with  $p=m-k,q=n-k$;

$\bullet$ if $\goth g=\goth{osp}(2m,2n)$ then $\goth g_x=\goth{osp}(2p,2q)$ 
with $p=m-k,q=n-k$.

In all cases it will be convenient to encode the core $\chi'$ by the
corresponding 
dominant typical weight $\mu+\rho'$ of $\goth g_x$. In what follows we write 
$$\chi':=\mu+\rho'=a_1\varepsilon_1+\dots+a_p\varepsilon_p+b_1\delta_1+\dots+b_q\delta_q,$$ 
where $a_i,b_j$ satisfy the additional assumptions of dominance and
typicality with respect to $\goth g_x$ (and additional positivity
condition for $\goth g_x\simeq \goth{osp}(2p,2q)$), more precisely 
\begin{equation}\label{poscond1}
a_1>\dots>a_p,b_1>\dots>b_q,
\end{equation}
\begin{equation}\label{poscond2}
a_i,b_j\in\mathbb Z, a_i\neq -b_j{ } \text{ if } \goth g=\goth {gl}(m,n),
\end{equation}
\begin{equation}\label{poscond3}
a_i,b_j\in \frac{1}{2}+\mathbb Z_{\geq 0}, a_i\neq b_j { } \text{ if } \goth g=\goth {osp}(2m+1,2n),
\end{equation}
\begin{equation}\label{poscond4}
a_i\in\mathbb Z_{\geq 0},b_j\in\mathbb Z_{> 0}, a_i\neq b_j{ } \text{
  if } \goth g=\goth {osp}(2m,2n).
\end{equation}
 We call the numbers $a_i,b_j$ the {\it {marks}} of the core.

Now we define $\goth g_{\chi}\subset \goth g$ corresponding to a
connected sub-Dynkin diagram containing the last node(s) of the
diagram of $\goth g$ in the following way:

$\bullet$ if $\goth g=\goth {gl}(m,n)$ then $\goth g_{\chi} \simeq \goth {gl}(k,k)$;

$\bullet$ if $\goth g=\goth{osp}(2m+1,2n)$ then $\goth g_{\chi} \simeq \goth{osp}(2k+1,2k)$;

$\bullet$ if $\goth g=\goth{osp}(2m,2n)$ and $a_p>0$,then $\goth g_{\chi} \simeq \goth{osp}(2k,2k)$;

$\bullet$ if $\goth g=\goth{osp}(2m,2n)$ and $a_p=0$,then $\goth g_{\chi} \simeq \goth{osp}(2k+2,2k)$.

Let $\goth p$ be the parabolic subalgebra containing $\goth b$ whose Levi part is
$\goth l=\goth g_{\chi}+\goth h$. 
Assume that $\chi$ is such that $\goth p$ is
admissible for $\chi$ and $\lambda$ is a dominant weight such that
$\chi_{\lambda}=\chi$.

\el

\noi {\bf Terminology} - If in addition $\lambda$ is $\goth p$-typical,
we call $\lambda$ {\it {stable}}.

\el

\noi Let $\mathcal {F}_{\leq \lambda}^{\chi}$ denote the
subcategory in $\mathcal {F}^{\chi}$ consisting of all $\goth g$-modules with
all weights $\leq \lambda$. It is not hard to see that if $\lambda$ is
stable then the highest weight of any simple $\goth g$-module in
$\mathcal {F}_{\leq \lambda}^{\chi}$
is also stable. By $\mathcal {F}_{\leq  \lambda}^{\chi}(\goth l)$
we denote the corresponding truncated category of $\goth l$-modules
with central character from the set $\Phi^{-1}(\chi)$ ($\Phi$ was defined just
before Lemma \ref{l4}).

Define the functors $\operatorname{Res}:\mathcal {F}^{\chi}\to \mathcal {F}^{\chi}(\goth l)$
and $\operatorname {Ind} :\mathcal {F}^{\chi}(\goth l)\to \mathcal {F}^{\chi}$ by 
$$\operatorname{Res} N=N^{\goth m}, \operatorname{Ind}M=\Gamma_0(G/P,M),$$
recall that $\goth m$ stands for the nilpotent radical of $\goth p$.

\begin{leme}\label{l3} - Assume that $\lambda$ is stable.
Then the functors $\operatorname{Res}$ and $\operatorname{Ind}$ establish 
an  equivalence between the categories $\mathcal {F}_{\leq\lambda}^{\chi}$ and
  $\mathcal {F}_{\leq\lambda}^{\chi}(\goth p)$. 
\end{leme}

\begin{proof}  - The statement easily follows from the typical
  lemma, (Lemma \ref{typ}). 
Indeed for any simple module $L_{\mu}(\goth  p)\in \mathcal {F}_{\leq\lambda}^{\chi}(\goth p)$,  
$\operatorname{Ind} L_{\mu}(\goth  p)=\Gamma_0(G/P,L_{\mu}(\goth p))=L_{\mu}$ is  the unique simple quotient of $\mathcal U(\goth g)\otimes_{\mathcal U(\goth p)}L_{\mu}(\goth p)$.
All the higher cohomology groups vanish. Therefore $\operatorname{Ind}$ is an exact functor which
  maps a simple module to a simple module. Clearly, 
$$\operatorname{Ind Res}(L_{\mu}(\goth p))\simeq L_{\mu}(\goth p), \operatorname{Res Ind}(L_{\mu})\simeq L_{\mu},$$
therefore the lemma holds. 
\end{proof}

Our next step in the proof of Theorem \ref{bl} is ``to move'' any
simple module to  $\mathcal {F}_{\leq \lambda}^{\chi}$ using translation functors.
Recall the definition of translation functors (see \cite{BG}):
let $V$ be a finite-dimensional $\goth g$-module. One defines a functor 
$T(V)^{\chi,\tau}:\mathcal {F}^{\chi}\to \mathcal {F}^{\tau}$ by
$T(V)^{\chi,\tau}(M)=(M\otimes V)^{\tau}$.
It is not difficult to see that $T(V)^{\chi,\tau}$ is exact and
$T(V^*)^{\tau,\chi}$ is left adjoint to $T(V)^{\chi,\tau}$. The
following lemma is also straightforward (see for example \cite{BG}).

\begin{leme}\label{tp} - If both $T(V^*)^{\tau,\chi}$ and $T(V)^{\chi,\tau}$
move a simple module to a simple module, then they establish an 
equivalence between the categories $\mathcal F^{\chi}$ and $\mathcal F^{\tau}$.
\end{leme}

\begin{leme}\label{tp1} - Let $\tau$ and $\chi$ be central
  characters. Assume 
  
$\bullet$ for any dominant $\mu$ with $\chi_{\mu}=\chi$ there exists a
unique weight $\gamma$ of $V$,  such that $\mu+\gamma$ is
dominant and $\chi_{\mu+\gamma}=\tau$;

$\bullet$ for any dominant $\nu$ with $\chi_{\nu}$=$\tau$ there exists a unique weight $\gamma'$ of $V^*$ such that $\nu+\gamma'$ is
dominant and $\chi_{\nu+\gamma'}=\chi$;

$\bullet$ the multiplicities of $\gamma$ in $V$ and  $\gamma'$ in
$V^*$ are 1.

Then $T(V)^{\chi,\tau}$ and $T(V^*)^{\tau,\chi}$ establish an  
equivalence between the categories $\mathcal F^{\chi}$ and $\mathcal F^{\tau}$.
\end{leme}

\begin{proof} - It suffices to prove that 
  $T(V)^{\chi,\tau}(L_{\mu})=L_{\mu+\gamma}$ for any $L_{\mu}\in \mathcal F^{\chi}$.

First we note that a $\goth b$-singular vector in
  $(L_{\mu}\otimes V)^{\tau}$ has weight $\mu+\gamma$. Hence
  $T(V)^{\chi,\tau}(L_{\mu})$ has a unique up to proportionality
  $\goth b$-singular vector. Since $T(V)^{\chi,\tau}(L_{\mu})$ is
  contragredient, it is either zero or $L_{\mu+\gamma}$. It is left to
  prove that $T(V)^{\chi,\tau}(L_{\mu})\neq 0$.

Let $C_{\mu}$ denote the one-dimensional $\goth b$-module of weight $\mu$.  Since
$$\Gamma_0(G/B,C_{\mu}\otimes V)=\Gamma_0(G/B,C_{\mu})\otimes V,$$
and $(C_{\mu}\otimes V)^{\Phi^{-1}(\tau)}$ has only one dominant
component $C_{\mu+\gamma}$, Corollary ~\ref{coroblocks} 
and Lemma ~\ref{l2}  imply the isomorphism
$$T(V)^{\chi,\tau}(\Gamma_0(G/B,C_{\mu}))\simeq \Gamma_0(G/B,C_{\mu+\gamma}).$$
Assume that $T(V)^{\chi,\tau}(L_{\mu})=0$. Then, since
$T(V)^{\chi,\tau}$ is exact, there must be a simple subquotient
$L_{\mu'}$ in $\Gamma_0(G/B,C_{\mu}))$ such that $T(V)^{\chi,\tau}(L_{\mu'})=L_{\mu+\gamma}$.
Therefore $\mu'+\delta=\mu+\gamma$ for some weight $\delta$ of $V$.
But the conditions of the lemma imply that $\mu'=\mu$. Contradiction.
\end{proof}

Let $E$ denote the standard $\goth g$-module. 

\begin{leme}\label{tp2} - Let ${\chi}$ be a central character with degree of atypicality $k>0$ and core   
$\chi'=a_1\varepsilon_1+...+a_p\varepsilon_p+b_1\delta_1...+b_q\delta_q$.
Let $V=E$ or $E^*$ and $\delta$ be a weight of $V$.
Assume that $\chi'+\delta$ satisfies the conditions
(\ref{poscond1})-(\ref{poscond4}) and in addition for 
$\goth g=\goth{osp}(2m,2n)$ the numbers of zero marks in $\chi'+\delta$
and in $\chi'$ are the same. Let $\tau$ be the central character such that 
$\tau'=\chi'+\delta$. Then $T(V)^{\chi,\tau}$ and $T(V^*)^{\tau,\chi}$
establish an equivalence between $\mathcal F^{\chi}$ and $\mathcal F^{\tau}$.
\end{leme}

\begin{proof} - Since $V=E$ or $E^*$, every weight $\gamma$ of $V$ has
  multiplicity 1, moreover $\gamma=\pm\varepsilon_i,\pm\delta_j$
  or $0$. The proof can be reduced to checking
  the conditions of Lemma 
  \ref{tp1} for any $L_\mu\in \mathcal F^{\chi}$. We will
  consider here the most tedious case of 
$\goth  g=\goth{osp}(2m,2n),\delta=\pm\varepsilon_i$, the
other cases are completely analogous and we leave them to the reader.

Since $\chi_{\mu}=\chi$ there are the following two possibilities:
either $(\mu+\rho,\varepsilon_j)=a_i$ for a unique $j\leq m$ or
$(\mu+\rho,\varepsilon_m)=-a_i$ (in the latter case
$(\mu+\rho,\varepsilon_j)>a_i$ for all $j<m$).

Let $\delta=\varepsilon_i$. In the former case take
$\nu=\mu+\varepsilon_j$ if $(\mu+\rho,\varepsilon_{j-1})>a_i+1$. If
$(\mu+\rho,\varepsilon_{j-1})=a_i+1$, there exists $l$ such that
$(\mu+\rho,\varepsilon_{j-1}+\delta_l)=0$ and one should take $\nu=\mu-\delta_l$. In
the latter case take $\nu=\mu-\varepsilon_m$ if $(\mu+\rho,\varepsilon_{m-1})>a_i+1$.
If $(\mu+\rho,\varepsilon_{m-1})=a_i+1$, there exists $k$ such that
$(\mu+\rho,\varepsilon_{m-1}+\delta_l)=0$, take
$\nu=\mu-\delta_l$.

Now deal similarly with the case $\delta=-\varepsilon_i$. In the former case take
$\nu=\mu-\varepsilon_j$ if $|(\mu+\rho,\varepsilon_{j+1})|< a_i-1$. If
$(\mu+\rho,\varepsilon_{j+1})=a_i-1$, there exists $l$ such that
$(\mu+\rho,\varepsilon_{j+1}+\delta_l)=0$ and choose
$\nu=\mu+\delta_l$. If
$(\mu+\rho,\varepsilon_{j+1})=-a_i+1$, then $j+1=m$, there exists $k$ such that
$(\mu+\rho,\varepsilon_{m}-\delta_l)=0$ and choose
$\nu=\mu+\delta_l$. 
Finally in the latter case take $\nu=\mu+\varepsilon_m$.
\end{proof}

\begin{leme}\label{tr} - Let $k$ be the degree of atypicality of
  $\chi$, 
$\chi'=a_1\varepsilon_1+...+a_p\varepsilon_p+b_1\delta_1...+b_q\delta_q$. 
Let $\goth p$ be the parabolic subalgebra with Levi part $\goth
l=\goth h+\goth g_{\chi}$ which contains $\goth b$. 
Let $\lambda$ be a dominant weight such that $\chi_{\lambda}=\chi$.
There exist a central character $\tau$ such that $\goth p$ is
admissible for $\tau$, and
a dominant stable weight $\mu$, such that $\chi_{\mu}={\tau}$ and
$\mathcal F^{\chi}_{\leq\lambda}$ is equivalent to $\mathcal F^{\tau}_{\leq\mu}$.
\end{leme}

\begin{proof} - Let 
$$\chi'=a_1\varepsilon_1+...+a_p\varepsilon_p+b_1\delta_1+...+b_q\delta_q.$$
If $a_1>b_1$ in $\goth{osp}$ case, $a_1>-b_q$ in $\goth{gl}$ case, put
$\chi'_1=\chi'+\varepsilon_1$, let $\chi_1$ be the central character with
core $\chi'_1$. In this way proceed to increase $a_1$ so that it is
bigger than the absolute value of any coordinate of $\lambda$ plus $p+q$.
If $a_1<b_1$ in the $\goth{osp}$ case, increase
$b_1$ in the same manner. In the $\goth{gl}$ case, if $a_1<-b_q$
decrease $b_q$. After this,
pick up the next mark in $\chi'$ and increase (decrease) it following
the same method to the absolute value of the previous mark -1. Proceed in the same manner with all marks of
$\chi'$ increasing the absolute value of each mark (except $a_p=0$ in the $\goth{osp}(2m,2n)$ case). Call the
resulting core $\tau'$, and let $\tau$ be the corresponding central character.
As follows from Lemma \ref{tp2} the categories $\mathcal F^{\chi}$ and
$\mathcal F^{\tau}$ are equivalent via a composition of translation
functors, which we denote by $T$. Then $T(L_{\lambda})=L_{\mu}$ and one
can easily check that $\mu$ is stable and $\goth p$-typical. Hence 
$\mathcal F^{\chi}_{\leq\lambda}$ is equivalent to $\mathcal F^{\tau}_{\leq\mu}$.

We would like to illustrate the above argument with few examples.

Let $\goth g=\goth{gl}(3,2)$, $\lambda+\rho=(2,0|3,0,-1)$, then $\mu+\rho=(7,-1|1,-5,-6)$.
 
Let $\goth g=\goth{osp}(5,4)$, $\lambda+\rho=(\frac{5}{2},-\frac{1}{2}|\frac{5}{2},\frac{3}{2},\frac{1}{2})$, then $\mu+\rho=(\frac{3}{2},-\frac{1}{2}|\frac{9}{2},\frac{3}{2},\frac{1}{2})$.

Let $\goth g=\goth{osp}(4,6)$, $\lambda+\rho=(4,-2|3,2,1)$, then $\mu+\rho=(7,-1|6,5,1)$.
\end{proof}

Lemma \ref{tr} and Lemma \ref{l3} imply Theorem \ref{bl}. 
Indeed, by Lemma \ref{tr} for any dominant $\lambda$ with
$\chi_{\lambda}=\chi$ the truncated category  $\mathcal F^{\chi}_{\leq\lambda}$
is equivalent to the ``stable'' truncated category 
$\mathcal F^{\tau}_{\leq  \mu}$ for a suitable choice of $\mu$.
The latter category is equivalent to $\mathcal F^{\tau}_{\leq  \mu}(\goth l)$
by Lemma \ref{l3}. Finally, $\mathcal F^{\tau}_{\leq  \mu}(\goth l)$
is equivalent to the truncated part of the most atypical block of
$\goth g_{\chi}$ since $\goth l$ is the direct sum of $\goth g_{\chi}$ and a
center. Since $\lambda$ is arbitrary one can extend this
equivalence to the whole $\mathcal F^{\chi}$. 

\section{Weight diagrams and translation functors}

In this section we define an alternative way to describe dominant
weights following Brundan and Stroppel. Their method allows one to
visualize the action of the translation functors defined in the
previous
section.

Let $\mathbb T\subset \mathbb R$ be a discrete set,
$X=(x_1,...,x_m)\in
\mathbb T^m$,
$Y=(y_1,...,y_n)\in\mathbb T^n$. A diagram 
$f_{X,Y}$ is a function defined on $\mathbb T$ whose values are 
multisets with elements $<,>,\times$ according to the following algorithm.

$\bullet$ Put the symbol $>$ in position  $t$ for all $i$ such that $x_i=t$.

$\bullet$ Put the symbol $<$ in position $t$ for all $i$ such that $y_i=t$.

$\bullet$ If there are both $>$ and $<$ in the same position replace them by 
the symbol $\times $, repeat if possible.

Thus, $f_{X,Y}(t)$ may contain at most one of the two symbols $>,<$.
We represent $f_{X,Y}$ by the picture with $0$ standing in position 
$t$ whenever $f(t)$ is an empty set.

Let  $\goth g=\goth{gl}(m,n)$. Let $\lambda$ be a dominant integral weight such that
$$\lambda+\rho=a_1\varepsilon_1+...+a_m\varepsilon_m+b_1\delta_1+...+b_n\delta_n.$$

Set $\mathbb T=\mathbb Z$, 
$$X_{\lambda}=(a_1,...,a_m), Y_{\lambda}=(-b_1,...,-b_n).$$ 
The diagram $f_{\lambda}=f_{X_{\lambda},Y_{\lambda}}$ is called the
{\it weight diagram} of $\lambda$.

A diagram is the weight diagram of some dominant weight if and only if
$f(t)$ is empty or is a just one element set since both sequences
$a_1,...,a_m$ and $b_1,...,b_m$ are strictly decreasing and hence do not have repetitions.

Each dominant weight is uniquely determined by its weight diagram.
The number of $<$ is $n$, the number of $>$ is $m$ (counting
$\times$ as both  $<$ and $>$). The number of $\times$ equals the
degree of atypicality. Replacing all $\times$ in the diagram by zeros gives a
diagram of the core.
For example, the diagram
$$\dots,<,\times,0,0,>,\times,\dots$$ 
where $\dots$  stand for empty positions and the
left $\times$ is at position $0$, corresponds to the weight
$$\lambda+\rho=(4,3,0|1,0,-4).$$

The translation functor $T(V)^{\chi,\tau}$ described in Lemma
\ref{tp2} moves a simple module $L_{\lambda}\in \mathcal F^\chi$ to
$L_{\mu}\in \mathcal F^{\tau}$ such that
$f_{\mu}$ is obtained from $f_{\lambda}$ by moving a symbol $<$ or $>$
at position $t$ 
to the next right position $t+1$ or to the next left position $t-1$ (the position $t$ 
and the direction are determined by a choice
of the core $\tau'$). Assume that the chosen direction is to the
right. If the next to the right position has $0$ or $\times$, we exchange
the symbols in position $t$ and $t+1$. For instance,
$$\dots,<,0,\dots\longrightarrow\dots,0,<,\dots$$
$$\dots, <,\times,\dots\longrightarrow\dots,\times,<,\dots.$$
The situation when the next to
the right symbol is $<$ or $>$ is forbidden by the conditions on
$\chi$ and $\tau$ (see Lemma ~\ref{tp2}). 
We move $<$ or $>$ to the left using the analogous rule.

Now for any dominant weight $\lambda$, let $\bar{\lambda}$ be the
corresponding weight in the equivalent most atypical block of $\goth g_{\chi}$. Then
$f_{\bar{\lambda}}$ is obtained from $f_{\lambda}$ by moving all
symbols $<$, $>$ to the right of all crosses by the procedures described above
and then replacing all of them by $0$. In our example
$f_{\bar{\lambda}}$ is 
$$\dots,\times,0,0,\times,\dots$$
with left $\times$ at position $-1$.

Note also that shifting a weight diagram by one one position to the right
corresponds to tensoring the corresponding module with the one
dimensional
representation of weight $(1,...,1|-1,...,-1)$.

Now let $\goth g=\goth{osp}(2m,2n)$. 
Set $\mathbb T=\mathbb Z_{\geq  0}$.
For a dominant weight $\lambda$ such that 
$\lambda+\rho=a_1\varepsilon_1+...+a_m\varepsilon_m+b_1\delta_1...+b_n\delta_n$
let 
$$X_{\lambda}=(|a_1|,...,|a_m|), Y_{\lambda}=(b_1,...,b_n), 
f_{\lambda}=f_{X_{\lambda},Y_{\lambda}}.$$

It is not difficult to see that $f_{\lambda}$ is a weight diagram of a
dominant $\lambda$ if and only if

$\bullet$ for any $t\neq 0$, $f_{\lambda}(t)$ is empty or just one element set;

$\bullet$ the multiset $f_{\lambda}(0)$ does not contain $<$, contains $>$ with
multiplicity at most 1 (it can contain any number of $\times$).

For example, if $\lambda=(2,0,0|3,0)$, then $f_{\lambda}={}_{\times}^{>},0,>,<,\dots$.
However, in this situation a weight is not uniquely determined by its
weight diagram. More precisely, if $f(0)\neq 0$, there is exactly one
weight with the weight diagram $f$, since all the coordinates of such a weight
are non-negative. If $f(0)=0$, then the coordinate $a$ corresponding
to the first $>$ or $\times$ can be chosen positive or negative.
For instance if
$$f=0,0,<,\times,>,\dots,$$
then the two weights $(4,3|3,2)$ and $(4,-3|3,2)$ are dominant and
have $f$ as their weight diagram.

\el

\noi {\bf Terminology} - To differenciate those two weights we call a
dominant weight {\it positive} if it does not have negative coordinates, and
{\it negative} otherwise.

\el

\noi The core of a weight can be obtained by replacing by $0$ all $\times$ in
the diagram. The translation functors from Lemma \ref{tp2} can be
described in the same way as in previous case, except that we do not allow
a symbol to move from or to the zero position. Indeed, if we want to move $>$
from the zero position, we can get two weights corresponding to the
same diagram, which means that the translation functor does not
provide an equivalence of blocks. Thus, in this case we
have two types of blocks, one with zero mark at its core ($>$ at the
zero position), and another without it. The former case corresponds to
$\goth g_{\chi}=\goth {osp}(2k+2,2k)$, the latter corresponds to $\goth g_{\chi}=\goth {osp}(2k,2k)$. 
Note that the atypicality degree $k$, as
before, is the number of $\times$ in a weight diagram. Finally, to get
the weight $\bar{\lambda}$ corresponding to $\lambda$ in the most atypical
block, as in the $\goth {gl}$ case, we move all $<,>$ to the right of all
$\times$ (except one at zero position) and then replace them by $0$.
A positive weight goes to a positive one, and a negative weight goes
to a negative one under this correspondence.

Below are two examples:

if $f_{\lambda}=0,0,<,\times,>,\times,\dots$, then
$f_{\bar{\lambda}}=0,0,\times,\times, \dots$;

if $f_{\lambda}={}_{\times}^{<},0,>,<,\dots$, then $f_{\bar{\lambda}}={}_{\times}^{<},0,\dots$.

Now let us discuss the case $\goth{osp}(2m+1,2n)$. We assume that
$\lambda$ is dominant and atypical, then all coordinates $a_i,b_j$ of
$\lambda+\rho$ belong to $-\frac{1}{2}+\mathbb Z_{\geq 0}$. Let 
$\mathbb T=\frac{1}{2}+\mathbb Z_{\geq 0}$ and define
$X_{\lambda}$, $Y_{\lambda}$ and $f_{\lambda}$ as in the case
$\goth g=\goth{osp}(2m,2n)$.
The dominance condition is equivalent to the following condition on a
weight diagram $f$

$\bullet$ $f(t)$ is empty or one element set for any $t\neq \frac{1}{2}$;

$\bullet$ $f(\frac{1}{2})$ may contain at most one of $<$ or  $>$ and any
number of $\times$. 

As in the previous case, it is possible that two dominant weights have
the same weight diagram. That may happen if $f(\frac{1}{2})$ does not contain
$>$ or $<$ and has at least one $\times$. For example,
the diagram with two $\times$ at $\frac{1}{2}$ corresponds to
$(\frac{1}{2},-\frac{1}{2}|\frac{1}{2},\frac{1}{2})$ and to $(-\frac{1}{2},-\frac{1}{2}|\frac{1}{2},\frac{1}{2})$. The translation functors, 
unlike in the previous case, mix those two
types of weights. So if the weight diagram has at least one $\times $
and no $<,>$ at the position $\frac{1}{2}$ we put an indicator (which we sometimes refer to as "sign") $\pm$ before the weight
diagram in parentheses. Its value is $+$ if the corresponding weight has the form 
$$\lambda+\rho=(a_1,...,a_{m-s},\frac{1}{2},-\frac{1}{2},...,-\frac{1}{2}|b_1,...,b_n),$$
and $-$ if the corresponding weight has the form
$$\lambda+\rho=(a_1,...,a_{m-s},-\frac{1}{2},-\frac{1}{2},...,-\frac{1}{2}|b_1,...,b_n),$$
where $s$ is the number of crosses at the position $\frac{1}{2}$.

The translation functors of Lemma ~\ref{tp2} act on the
diagrams, as in the previous case. The only difference is that one
allows to move $<$ or $>$ from the position $\frac{1}{2}$
to the right but such move transforms a diagram without indicator
to one with it. 
If $f(\frac{3}{2})=0$, then the indicator of the new diagram is $-$, 
if $f(\frac{3}{2})=\times$, then the indicator is $+$.
For example
$${}_{\times}^{<},0,\dots\longrightarrow (-)\times,<,\dots$$

$${}_{\times}^{<},\times,\dots\longrightarrow (+){}_{\times}^{\times},<,\dots$$

Moving $<$ (resp. $>$) at the position $\frac{3}{2}$ to the left is possible if $f(\frac{1}{2})$ does not have
$<$ or $>$ already, and therefore either $f(\frac{1}{2})=0$ or $f(\frac{1}{2})$ must have an
indicator. If
the indicator of $f$ is $-$ we just move $<$ (resp. $>$)
to the position $\frac{1}{2}$ and put $f(\frac{3}{2})=0$. If the indicator of $f$ is
$+$, one should exchange $<$ (resp. $>$) at the position $\frac{3}{2}$ with one $\times$ at the
position $\frac{1}{2}$.
For example,
$$(-){\times},<,0,\dots\longrightarrow {}_{\times}^<,0,\dots$$
$$(+){\times},<,0,\dots\longrightarrow <,\times,\dots$$

To get the weight $\bar{\lambda}$ in the most atypical block
corresponding to $\lambda$, one does the same as in two previous cases
(one moves all $<$ and $>$ to the right of
all crosses and then replaces them by $0$). 

\begin{rem} - It is clear from above that the
  most atypical blocks with trivial central characters of Lie
  superalgebras $\goth{osp}(2k+1,2k+2)$ and $\goth{osp}(2k+1,2k)$ are
  equivalent.
\end{rem}

\section{Reduction to the most atypical case}

\begin{leme}\label{acycl} - Let $\alpha$ be a simple root of $\goth  b_0$ such that 
either $\alpha$ is a simple root of $\goth b$ or $\alpha$ is a sum of two
isotropic simple roots $\alpha_1+\alpha_2$. Let $\goth q$  be the
parabolic subalgebra with Levi part $\goth l$ containing $\goth b$ such that
$\alpha$, respectively $\alpha_1,\alpha_2$ are orthogonal to all roots
in $\Delta(\goth l)$. 
Assume that
$\nu$ is a $\goth l$-dominant weight, such that
$(\nu+\rho,\alpha)=0$ (and $(\nu+\rho,\alpha_j)\neq 0$ at least for one
$j$ in the second case). Then $\Gamma_i(G/Q,L_{\nu}(\goth q))=0$ for all
$i\geq 0$.
\end{leme}

\begin{proof} - Consider the parabolic subalgebra $\goth p$ obtained
from $\goth q$ by adding the roots $-\alpha$, ($-\alpha_1,-\alpha_2$
in the second case). The fibres of the canonical projection 
$\pi: G/P \to G/Q$ are isomorphic to $G'/B'$, where $\goth g'$ is
isomorphic $\goth{sl}(2)$ in the first case and $\goth{sl}(1,2)$ in
the second case, $\goth b'=\goth g' \cap \goth b$. We claim that $\mathcal L_{\nu }(\goth p)^*_{| fibre}$
is acyclic. Indeed, in the first case 
$\mathcal L_{\nu }(\goth p)^*_{|fibre}$ is an invertible sheaf on
$\mathbb P^1$, and the condition $(\nu+\rho,\alpha)=0$ immediately implies
that it is acyclic. In the second case the underlying variety
$(G'/B')_0$ is 
isomorphic to $\mathbb P^1$. To calculate the cohomology we use Penkov's remark. 
The sheaf $\tilde{\mathcal L}_{\nu }(\goth p)^*_{| fibre}$ has a
filtration with four simple terms, one dominant 
$\mathcal O_{\nu }$, two acyclic terms $\mathcal O_{\nu-\alpha_1}$, $\mathcal O_{\nu-\alpha_2}$,
and one antidominant $\mathcal O_{\nu-\alpha_1-\alpha_2}$. If the cohomology groups of
$\mathcal L_{\nu }(\goth p)^*_{| fibre}$ are non-trivial, they must
be one-dimensional. But that would imply
$(\nu+\rho,\alpha_1)=(\nu+\rho,\alpha_2)=0$.
Contradiction.

Now the statement follows from Leray spectral sequence.
\end{proof}  

\begin{leme}\label{reduction} - Let $\goth g=\goth{gl}(m,n)$. Let $\lambda$ be a dominant weight and
$\chi=\chi_{\lambda}$. Let $V,\delta$ and $\tau$ satisfy all
the  conditions of Lemma ~\ref{tp2},
$T(V)^{\chi,\tau}(L_{\lambda})=L_{\mu}$.
Then $$\Gamma_i(G/B,L_{\mu}(\goth b))=T(V)^{\chi,\tau}(\Gamma_i(G/B,L_{\lambda}(\goth b))).$$
\end{leme}

\begin{proof} - By Corollary ~\ref{coroblocks}, one has
$$\Gamma_i(G/B,(V\otimes L_{\lambda}(\goth b))^{\Phi^{-1}(\tau)})=T(V)^{\chi,\tau}(\Gamma_i(G/B,L_{\lambda}(\goth b))).$$
We note that $(V\otimes L_{\lambda}(\goth b))^{\Phi^{-1}(\tau)}$ has a
filtration with simple quotients $L_{\nu}(\goth b)$ such that all
$\nu\neq\mu$ are not dominant and satisfy the conditions of Lemma
~\ref{acycl}. Hence $\Gamma_i(G/B,L_{\nu}(\goth b))=0$ for all
$\nu\neq\mu$. The statement follows.
\end{proof}

\begin{coro}\label{cor21} - Let $\goth g=\goth{gl}(m,n)$. 
Let $\lambda$ be a dominant weight with atypicality
  degree $k$, $T$ be the functor which establishes an equivalence between
  $\mathcal F^{\chi_{\lambda}}$ and the most atypical block of $\goth  g_{\chi}$. 
Denote by $\bar{\lambda}$ the highest weight of $T(L_{\lambda})$. If
  $B_{\chi}=B\cap G_{\chi}$, then
$$K_{G,B}^{\lambda,\mu}(z)=K_{G_{\chi},B_{\chi}}^{\bar{\lambda},\bar{\mu}}(z)$$
\end{coro}
\begin{proof} - Due to Lemma ~\ref{reduction} it is sufficient to prove
  the statement for the case of stable $\lambda$ (see Lemma
  ~\ref{l3}). Let $\goth p=\goth g_{\chi}+\goth b$. Consider the natural
  projection $\pi:G/B\longrightarrow G/P$, the fibres of $\pi$ are
  isomorphic to $P/B=G_{\chi}/B_{\chi}$. We note that all the simple
  subquotients in
  $R^q\pi^*(\mathcal L_{\lambda}(\goth b)^*_{|fibre})$ are stable,
  hence $H^i(G/P,R^q\pi^*(\mathcal L_{\lambda}(\goth
  b)^*_{|fibre}))=0$ for $i>0$. The statement follows from Leray
  spectral sequence.
\end{proof}
If $\goth g=\goth{gl}(m,n)$, the above corollary reduces the calculation of 
the cohomology groups of an invertible sheaf $\mathcal O_{\lambda}$ 
on the flag variety $G/B$ (with dominant $\lambda$) to the
situation where $G=G_{\chi}$ and $\lambda$ has the same central character as
the trivial module. If $\goth g=\goth{gl}(m,n)$ it is known that 
$\Gamma_i(G/B,L_{\lambda}(\goth b))=0$ for all $i>0$, and
$\Gamma_0(G/B,L_{\lambda}(\goth b))$ is the Kac module
$K_{\lambda}$. Hence, calculating
$K_{G,B}^{\lambda,\mu}(z)={}^0K_{G,B}^{\lambda,\mu}$ provides an
algorithm for finding the characters of simple modules. That was done
in \cite{V.Sel} and we do not repeat these calculations here.

\el

Instead we concentrate on the case of the orthosymplectic group.

To do so, we have to use Theorem ~\ref{ind}, and unfortunately we
are only able to calculate the Euler characteristic
$K_{G,B}^{\lambda,\mu}(-1)$. That would be sufficient 
to find the characters if we had an assertion like "$K_{G,B}^{\lambda,\lambda}(-1)=1$ and
$K_{G,B}^{\lambda,\mu}(-1)\neq 0$
implies $\mu \leq \lambda$". 

However, in general this is not true, moreover,
in some cases $K_{G,B}^{\lambda,\mu}(-1)=0$. So we should substitute 
$B$ by some larger parabolic subgroup $Q_{\lambda}$ which depends on $\lambda$.  

Let $\lambda$ be a dominant weight. Denote by $\goth g_{\lambda}$ the
subalgebra of $\goth g$ defined by the Dynkin subdiagram
corresponding to the simple roots of $\goth g$ such
that all the coordinates of ${\lambda}$ restricted to 
$\goth g_{\lambda}$ are zero and $\goth g_{\lambda}$ is isomorphic to
$\goth {osp} (2s,2s),\goth {osp} (2s+2,2s)$ or $\goth {osp}(2s+1,2s)$. 

\el

\noi {\bf Terminology} - We call this subalgebra $\goth g_{\lambda}$ the
{\it tail} subalgebra
of the weight $\lambda$ and the module $L_{\lambda}$, 
and call $s$ the {\it length of the tail} of $\lambda$.

\el

\noi The
reader can check that $s$ is the number of $\times$'s at $0$
(respectively $\frac{1}{2}$) in the weight diagram except the case when the
indicator has $+$. In the latter case $s$ is the number of $\times$ at
$\frac{1}{2}$ minus $1$. Let $\goth q_{\lambda}$ be the parabolic subalgebra
with Levi part $\goth g_{\lambda}$. 

\begin{leme}\label{reduction1} - Let $\lambda$ be a dominant weight and
$\chi=\chi_{\lambda}$. Let $V,\delta$ and $\tau$ satisfy all
the  conditions of Lemma ~\ref{tp2}. Let 
$T(V)^{\chi,\tau}(L_{\lambda})=L_{\mu}$.
Then $Q_{\lambda}=Q_{\mu}$ and  
$$\Gamma_i(G/Q_{\lambda},L_{\mu}(\goth q_{\lambda}))=T(V)^{\chi,\tau}(\Gamma_i(G/Q_{\lambda},L_{\lambda}(\goth q_{\lambda}))).$$
\end{leme}

\begin{proof} - 
The same as of Lemma ~\ref{reduction}.
\end{proof}

\begin{leme}\label{tailbl} - The functor $T$ from $\mathcal F^{\chi}$
  to the most atypical block of $\goth g_{\chi}$
  containing trivial module, which
  provides an equivalence of categories, preserves the tails of
  simple modules.
\end{leme}
\begin{proof} - Straightforward. 
\end{proof}

For a dominant weight $\lambda$ let $\bar{\lambda}$ be the weight of
$T(L_{\lambda})$. Let $\goth q_{\lambda}$ be the parabolic subalgebra
of $\goth g$ whose Levi part is $\goth h+\goth g_{\lambda}$.
The following corollary can be proved exactly as Corollary ~\ref{cor21}.

\begin{coro}\label{cor22} - Let $\lambda$ be a dominant weight with
  central character $\chi$. Then
$$K_{G,Q_{\lambda}}^{\lambda,\mu}(z)=K_{G_{\chi},Q_{\bar{\lambda}}}^{\bar{\lambda},\bar{\mu}}(z).$$
\end{coro}

The above corollary reduces the calculation of
$K_{G,Q_{\lambda}}^{\lambda,\mu}(-1)$ to the case where $\goth g$ is one of 
$\goth {osp} (2k,2k),\goth {osp} (2k+2,2k), \goth {osp}(2k+1,2k)$,  
and $\lambda$ and $\mu$ have the trivial central character.

\section{Recursion}

\noi In this section we assume that $\goth g$ is either
$\goth {osp} (2k,2k),\goth{osp}(2k+2,2k)$ or $\goth {osp} (2k+1,2k)$.

\noi Then all the simple roots are odd and the Dynkin diagram of $\goth g$ is
either
$$\otimes-...-\otimes{}^{\otimes}_{\otimes}|,$$
(with $2k$ vertices if $\goth g=\goth {osp}(2k,2k)$ and $2k+1$
vertices if $\goth g=\goth {osp}(2k+2,2k)$) or
$$\otimes-...-\otimes\rightarrow{\bullet}$$
(with $2k$ vertices for $\goth g=\goth {osp}(2k+1,2k)$).
The corresponding Borel subalgebra is called {\it mixed} (\cite{G-L}).

Let $\goth s^1$ be the subalgebra of $\goth g$ generated by all the simple
roots of $\goth g$ except 
the first two, $\goth s^2$ be the subalgebra of $\goth s^1$ generated
by all the simple roots of $\goth s^1$ except the first two etc... Each time
$\goth s^i$ has the same type as $\goth g$. Let $\goth p^i$ be the
parabolic subalgebra with Levi subalgebra $\goth l^i=\goth h+\goth
s^i$ for $i\leq k-1$, and let  
$\goth p^k=\goth b$. We have a flag of parabolic subalgebras  
\begin{equation}\label{flag}
\goth g\supset\goth p^1\supset\dots\supset\goth p^k=\goth b.
\end{equation}

In this section, we assume that the simple finite-dimensional
module $L_{\lambda}$ has the
same central character as the trivial module, and we denote this
central character by $\chi$. In the case of $\goth{osp}(2k,2k)$, we
also assume that $\lambda$ is positive. That implies
$$\lambda+\rho=a_1\varepsilon_1+...+a_k\varepsilon_k+b_1\delta_1+...+b_k\delta_k,$$ 
Moreover, $|a_i|=|b_i|$ for all $i\leq k$. The weight diagram of
$\lambda$ does not have symbols $>,<$  except $>$ at $0$ for $\goth g=\goth{osp}(2k+2,2k)$.

\begin{rem} - 
Let $\goth g=\goth{osp}(2k,2k)$ and 
$$\lambda=a_1(\varepsilon_1+\delta_1)+...+a_k(\varepsilon_k+\delta_k),$$
define 
$$\lambda'=a_1(\varepsilon_1+\delta_1)+...+a_k(-\varepsilon_k+\delta_k).$$
If $\sigma$ is the automorphism of $\goth g$ induced by the symmetry of the
Dynkin diagram, and $M^{\sigma}$ is the module obtained from $M$ by twisting by
$\sigma$, then $L_{\lambda}^{\sigma}=L_{\lambda'}$. Since $\sigma$ acts on $G/P^1$, one
has 
$$K_{G,P^1} ^{\lambda , \mu} (z)=K_{G,P^1} ^{\lambda' , \mu'} (z).$$

Thus, if we know $K_{G,P^1} ^{\lambda , \mu} (z)$ for all positive
$\lambda$, we can easily obtain them for all $\lambda$.
\end{rem}

In this section, we give a recursion procedure to compute the polynomials 
$K^{\lambda , \mu}_{G,P^1}(z)$. This recursion is double: if $\lambda$
is "far from tail and far from the even walls" (Proposition
\ref{pt1}), we get it "closer to the tail and the walls". If $\lambda$
is close to the wall and far from the tail (Proposition \ref{pt2}) we
decrease the rank of the Lie superalgebra. 
Finally, when $\lambda$ is very close to the tail, we compute the
cohomology directly (Propositions \ref{pt3}, \ref{pt4} in
the next section).

For a Laurent polynomial $F(z)\in \mathbb C[z,z^{-1}]$, we denote by
$F(z)_+$ the polynomial obtained from $F(z)$ by removing the monomials
with negative powers of $z$.

\el

\noi {\bf Notation} - Let $\alpha=\varepsilon_1+\delta_1$.

\el

Before we start, we prove several technical statements that will be needed
later.

Recall that $E$ denotes the standard $\goth g$-module.
\begin{leme}\label{lp00} - Let $\tau$ be a dominant central character
  with degree of atypicality $k-1$. Then 
  $(L_{\lambda}\otimes E)^{\tau}$ is either simple or zero.

Let $\beta=\varepsilon_i+\delta_i$, $\lambda-\beta$ is dominant and
  $a_i>\frac{3}{2}$. Let
  $\tau=\chi_{\lambda-\delta_i}$ or $\chi_{\lambda-\varepsilon_i}$. 
Then 
  $(L_{\lambda}\otimes E)^{\tau}=0$.

\end{leme}
\begin{proof} - There exists at most one weight
  $\gamma$ of $E$ such that $\lambda+\gamma$ is dominant and
  $\chi_{\lambda+\gamma}=\tau$. The best way to see it is via weight
  diagrams. Indeed, the weight diagram of $\lambda+\gamma$ is obtained
  from that of $\lambda$ by ``separating'' one $\times$ in two halves
  $>,<$ and moving one half one position to the left or to the
  right. It is clear that in this way one can get at most one dominant
  weight diagram   with given core. Hence the first statement.

To prove the second statement, assume the opposite, say, $(L_{\lambda}\otimes E)^{\chi}=L_{\lambda-\delta_i}.$
Let $M_{\pi}$ denote the Verma module with highest weight $\pi$. 
Then $(M_{\lambda}\otimes E)^{\tau}$ has a filtration by Verma modules
with highest weights $\lambda-\gamma$ for all weights $\gamma$ of $E$
such that $\chi_{\lambda-\gamma}=\tau$. By direct inspection, 
$\lambda-\delta_i$ is the only such 
weight. Therefore  $(M_{\lambda}\otimes E)^{\tau}=M_{\lambda-\delta_i}$. Similarly,  
$(M_{\lambda-\beta}\otimes E)^{\tau}$ has a filtration by Verma
modules, one of the terms of this filtration is
$M_{\lambda-\delta_i}$. Thus, $L_{\lambda-\delta_i}$ occurs in
$(M_{\lambda-\beta}\otimes E)^{\tau}$.

It is known (one can find a proof in \cite{Gor}) that 
$$\operatorname{Hom}_{\goth g}(M_{\lambda-\beta},M_{\lambda})\neq 0.$$  
Consider the exact sequence
$$0\rightarrow S \rightarrow M_{\lambda-\beta}\rightarrow M_{\lambda}\rightarrow F\rightarrow
0.$$
Apply the translation functor $T(E)^{\chi,\tau}$ to it
$$0\rightarrow (S\otimes E)^{\tau} \rightarrow (M_{\lambda-\beta}\otimes E)^{\tau}\rightarrow
M_{\lambda-\delta_i}\rightarrow (F\otimes E)^{\tau}\rightarrow 0.$$
 Since all weights of $S$ are strictly less than $\lambda -\beta$,
$\lambda-\delta_i$ is not a weight of $(S\otimes E)^{\chi} $, hence the latter
does not have a simple component $L_{\lambda-\delta_i}$. Therefore, 
by above, the multiplicity of $L_{\lambda-\delta_i}$ in 
$(F\otimes E)^{\tau}$ is zero. But 
$(F\otimes E)^{\tau}$ is a highest weight module with highest weight
$\lambda-\delta_i$. Therefore $(F\otimes E)^{\tau}=0$. Since
$L_{\lambda}$ is a quotient of $F$, we get $(L_{\lambda}\otimes E)^{\tau}=0$. 

Lemma is proven.
\end{proof}

The following lemma is very important in our calculations. We will use it 
in induction step to reduce the rank of $\goth g$.

\begin{leme}\label{lp0} - Let $\tau\neq \chi$ be a dominant central
  character with atypicality degree $k$ or
  $k-1$, $\kappa$ be a dominant weight  with
  central character $\tau$ (see section
  5 for definition). Assume that $(\kappa+\rho,\alpha)=0$. Let $T$ be the
 functor establishing an equivalence
  between $\mathcal F^{\tau}$ and the maximal atypical block of
  $\goth g_{\tau}$. Let $T(L_{\mu})=L_{\bar{\mu}}(\goth g_{\tau})$. Then
$$K^{\kappa,\mu}_{G,P^1}(z)=K^{\bar{\kappa},\bar{\mu}}_{G_{\tau},P^1_{\tau}}(z),$$
where $P^1_{\tau}\subset G_{\tau}$ is the analogue of $P^1$ for $G_{\tau}$.
\end{leme}

\begin{proof} - The weight diagram $f_{\kappa}$ has $\times$ at the 
  rightmost non-empty position, and symbols $<$ and $>$ somewhere to the
  left of this $\times$. Using translation
  functors one can move those symbols next to the rightmost
  $\times$. (If one of those $<,>$ is at position $0$ we don't
  move it, 
  just move the second symbol.) Those translation functors commute with
  $\Gamma_i(G/P^1,\bullet)$ as their action on simple $\goth p^1$-modules
  is the same as on the corresponding $\goth g$-modules. Hence we can
  assume without loss of generality that $<$ and $>$ stand next to the
  rightmost $\times$ of $\kappa$. 
  Let $\goth q=\goth p^1_{\tau}+\goth b$. We claim that, in this
  case, 
  $\Gamma_i(G/P^1,L_{\kappa}(\goth
  p^1))=\Gamma_i(G/Q,L_{\kappa}(\goth q))$.
  Indeed, this follows
  from the fact that the restriction of $\kappa$ on $\goth g_{\tau}$ is
  typical with respect to 
  $\goth p^1_{\tau}=\goth q\cap\goth g_{\tau}$, and by the typical
  lemma, the Leray spectral sequence for the
  canonical projection $\pi:G/Q \longrightarrow G/P^1$ degenerates,
  ($R^q\pi^*\mathcal L_{\kappa}^*=0$ for all $q>0$).
Now consider the translation functor which moves $<$ and $>$ to the
right of the rightmost $\times$. This functor commutes with
$\Gamma_i(G/Q,\bullet)$,
since it moves a simple $\goth q$-module to
the one whose simple subquotients are all acyclic except one
(as in the proof of
  Lemma ~\ref{reduction}). To finish the proof
  use stability of the weight obtained from $\kappa$ by translations
and proceed as in the proof of Corollary ~\ref{cor21}.
\end{proof}

Let $\goth g=\goth{osp}(2k+1,2k)$. Let $\lambda'=\lambda$ if the
diagram of $\lambda$ does not have a sign. If the weight diagram of $\lambda$
has a sign let $\lambda'$ denote the weight whose diagram is obtained
from that of $\lambda$ by switching the sign. For example, if
$\lambda=0$, then $\lambda'=\varepsilon_1$.

Recall that $\Phi$ is defined just before Lemma ~\ref{l4}.

\begin{leme}\label{funsw} - Let $\goth g=\goth{osp}(2k+1,2k)$. Then
  $T(E)^{\chi,\chi}$ is an equivalence of categories, and
$T(E)^{\chi,\chi}(L_{\lambda})=L_{\lambda'}$.
\end{leme}
\begin{proof} - The only dominant weights with central character $\chi$
  of the form $\lambda+\gamma$, $\gamma$ being a weight of $E$, are
  $\lambda$ and $\lambda'$. It suffices to show then that  
$T(E)^{\chi,\chi}(L_{\lambda})=L_{\lambda'}$ for $\lambda'\neq  \lambda$.

A direct calculation proves the statement for $E$ ($\lambda=\varepsilon_1$) and for the trivial
  module ($\lambda=0$). Choose a maximal parabolic subalgebra $\goth  p$ with
  semi-simple part $\goth s$ such
  that $L_{\lambda}(\goth s)$ is either standard or trivial.  
Then 
$$(L_{\lambda}(\goth p)\otimes
 E)^{\Phi^{-1}(\chi)}=L_{\lambda'}(\goth p),$$ 
and 
\begin{equation}\label{uf}
(\Gamma_0(G/P,L_{\lambda}(\goth p))\otimes E)^{\chi}=\Gamma_0(G/P,L_{\lambda'}(\goth p)).
\end{equation}
Note that $\Gamma_0(G/P,L_{\lambda'}(\goth p))$
does not contain any subquotient isomorphic to $L_{\lambda}$, because
the corresponding parabolically induced module does not have
$L_{\lambda}$ as a subquotient.
Therefore the exact sequence 
$\Gamma_0(G/P,L_{\lambda}(\goth p))\to L_{\lambda}\to 0$, after application of $T(E)^{\chi,\chi}$, becomes 
$\Gamma_0(G/P,L_{\lambda'}(\goth p))\to L_{\lambda'}\to 0$. That implies the statement.
\end{proof}

We call $T(E)^{\chi,\chi}$ the {\it switch} functor. Lemma
~\ref{indgen} and Lemma ~\ref{funsw} imply that for any $\lambda\neq 0,\varepsilon_1$
\begin{equation}\label{sw1}
K^{\lambda,\mu}_{G,P^1}(z)=K^{\lambda',\mu'}_{G,P^1}(z).
\end{equation}

\begin{prop}\label{pt1} - Let $a_1>\frac{3}{2}$ and $a_1 > |a_2|+1$ if $k>1$. Then 

i) $K_{G,P^1} ^{\lambda , \mu} (z) = (z^{-1}K_{G,P^1} ^{\lambda -
  \alpha , \mu}(z))_+$ for 
$\mu \notin \{\lambda , \lambda - \alpha\}$,

ii) $K_{G,P^1} ^{\lambda , \lambda} (z) = 1$

iii) $K_{G,P^1} ^{\lambda , \lambda - \alpha} (z) = 1$.

\end{prop}

To prove the proposition we start with the following

\begin{leme}\label{lp1} - Let $\lambda$ be as in the proposition, $\nu= \lambda-\varepsilon_1$.
One has the
  following short exact sequence of $\goth p^1$-modules:
$$0\rightarrow L_{\lambda}(\goth p ^1) \rightarrow (L_{\nu}(\goth p ^1) \otimes E)^{\Phi^{-1}(\chi)} \rightarrow L_{\lambda - \alpha}(\goth p ^1) \rightarrow 0.$$
\end{leme}

\begin{proof} - Consider the case $\goth g=\goth{osp}(2k+2,2k)$ or
  $\goth{osp}(2k+1,2k)$.
Note that $\goth p^1$-irreducible subquotients
  of $E$ are $L_{\pm\varepsilon_1}(\goth p^1),L_{\varepsilon_2}(\goth  p^1)$ 
and $L_{\pm\delta_1}(\goth p^1)$. 
All of them except $L_{\varepsilon_2}(\goth  p^1)$ are one-dimensional.
It is easy to check that the only 
$\goth  p^1$-dominant weights of the form $\nu+\gamma$ ($\gamma$
  being a weight of $E$) with central character $\chi$ are
  $\lambda$ and $\lambda-\alpha$. Hence we have
  $$(L_{\nu}(\goth p ^1)\otimes L_{-\varepsilon_1}(\goth
  p^1))^{\Phi^{-1}(\chi)}=(L_{\nu}(\goth p ^1)\otimes
  L_{\varepsilon_2}(\goth  p^1))^{\Phi^{-1}(\chi)}=$$
  $$=(L_{\nu}(\goth p ^1)\otimes L_{\delta_1}(\goth  p^1))^{\Phi^{-1}(\chi)}=0.$$
On the other hand, 
$$L_{\nu}(\goth p ^1)\otimes L_{-\delta_1}(\goth
p^1)=L_{\lambda-\alpha}(\goth p^1),$$
$$L_{\nu}(\goth p ^1)\otimes L_{\varepsilon_1}(\goth
p^1)=L_{\lambda}(\goth p^1).$$
The exact sequence follows immediately.

The case $\goth g=\goth{osp}(2k,2k)$  can be
done in the same way with the 
substitution of $\delta_2$ in place of $\varepsilon_2$.
\end{proof}

The short exact sequence of Lemma ~\ref{lp1} leads to the following long exact sequence:

 $$\dots \rightarrow \Gamma _1 (G/P^1,  L _{\lambda - \alpha}(\goth p
 ^1)) \rightarrow  \Gamma _0 (G/P^1,  L _{\lambda}(\goth p ^1))
 \rightarrow  \Gamma _0 (G/P^1, ( L _{\nu}(\goth p ^1) \otimes  
E)^{\Phi^{-1}(\chi)})\rightarrow $$
$$\rightarrow\Gamma _0 (G/P^1,  L _{\lambda - \alpha}(\goth p ^1)) \rightarrow 0.$$
By Corollary ~\ref{coroblocks} we have
$$\Gamma _i (G/P^1, ( L _{\nu}(\goth p ^1) \otimes E)^{\Phi^{-1}(\chi)}) = 
(\Gamma _i (G/P^1, L _{\nu}(\goth p ^1)) \otimes E)^{\chi}.$$

Now note that $\nu$ is $\goth p^1$-typical. Therefore 
$\Gamma _i(G/P^1, (L _{\nu}(\goth p ^1))=0$ for all $i>0$, and 
$\Gamma _0 (G/P^1, (L _{\nu}(\goth p ^1))=L_{\nu}$. Thus, the exact
sequence degenerates into the following
\begin{equation}\label{r1}
\Gamma _i (G/P^1,  L_{\lambda - \alpha}(\goth p ^1)) \simeq 
\Gamma _{i-1} (G/P^1,  L_{\lambda}(\goth p ^1)) \; for \; i\geq 2,
\end{equation}
and
\begin{equation}\label{r2}
0\rightarrow \Gamma _1 (G/P^1,  L _{\lambda - \alpha}(\goth p ^1)) \rightarrow  \Gamma _0 (G/P^1,  L _{\lambda}(\goth p ^1)) \rightarrow ( L _{\nu}\otimes  E)^{\chi}\rightarrow \Gamma _0 (G/P^1,  L _{\lambda - \alpha}(\goth p ^1)) \rightarrow 0.
\end{equation}
Now we concentrate on the last exact sequence.

Denote by $V_{\lambda}$
the image of 
$\Gamma _0 (G/P^1,  L _{\lambda}(\goth p ^1))$ in $( L
_{\nu}\otimes  E)^{\chi}$.
Since $V_{\lambda}$ is isomorphic to a quotient of $\Gamma _0 (G/P^1,
L _{\lambda}(\goth p ^1))$, Lemma ~\ref{l2} implies that $V_{\lambda}$
is generated by a highest vector of weight $\lambda$. Lemma
~\ref{lp4} below explains the structure of $V_{\lambda}$.
\begin{leme}\label{lp2} - The module $( L_{\nu}\otimes  E)^{\chi}$ is a
    contragredient $\goth g$-module with a unique irreducible submodule
    and a unique irreducible quotient isomorphic to $L_{\lambda-\alpha}$.
\end{leme}
\begin{proof} - The module $( L_{\nu}\otimes  E)^{\chi}$ is  contragredient because
the tensor product of contragredient modules is contragredient and all the
simple modules are contragredient. The only dominant weights of the
form $\nu+\gamma$ ($\gamma$ being a weight of $E$) with central
character $\chi$ are $\lambda$ and $\lambda-\alpha$. Hence they are the
only possible weights of $\goth b$-singular vectors in $( L_{\nu}\otimes  E)^{\chi}$.
The exact sequence ~\ref{r2} implies that $( L_{\nu}\otimes E)^{\chi}$ 
has simple subquotients isomorphic to $L_{\lambda}$ and $L_{\lambda-\alpha}$.

Assume that the socle of $( L_{\nu}\otimes  E)^{\chi}$ contains a
$\goth b$-singular vector of weight $\lambda$. This gives an inclusion
$L_\lambda \subset ( L_{\nu}\otimes  E)^{\chi}$ and, by duality, a
surjection
$( L_{\nu}\otimes  E)^{\chi} \rightarrow L_\lambda$. Since there is
only one vector of weight $\lambda$, the composition of those two maps
must be the identity, thus the exact sequence \ref{r2} splits. Hence
$\Gamma _0 (G/P^1,  L _{\lambda}(\goth p ^1))$ is contragredient too,
and since it's generated by a $\goth b$-singular vector of weight
$\lambda - \alpha$, it must be equal to $L_{\lambda - \alpha}$. Thus,
if $( L_{\nu}\otimes  E)^{\chi}$ contains $L_\lambda$, it is equal to
$L_\lambda \oplus L_{\lambda - \alpha}$.

Let us show that this is impossible. 
Assume the opposite. Then
$$\mathbb C=\operatorname{Hom}_{\goth g}(L_{\lambda},L_{\nu}\otimes
E)=\operatorname{Hom}_{\goth g}(L_{\lambda}\otimes E,L_{\nu}),$$
hence $(L_{\lambda}\otimes E)^{\chi_{\nu}}=L_{\nu}$. But Lemma ~\ref{lp00}
(with $\tau=\chi_{\nu}$) implies   
$(L_{\lambda}\otimes E)^{\chi_{\nu}}=0$. 
Contradiction.

So now, we are sure that the socle of $( L_{\nu}\otimes  E)^{\chi}$ is
isomorphic to $L_{\lambda - \alpha}$; by duality, the same
holds for the head, which is a quotient of $\Gamma _0 (G/P^1,  L
_{\lambda}(\goth p ^1))$.
\end{proof}   

\begin{leme}\label{lp3} - The only weight $\mu$ such that $L_{\mu}$ may occur
  as a simple subquotient in both $\Gamma _0 (G/P^1,  L _{\lambda}(\goth  p^1))$
and $\Gamma _0 (G/P^1,  L _{\lambda-\alpha}(\goth p^1))$ is $\lambda - \alpha$.
\end{leme}
\begin{proof} - We prove the lemma by induction on the dimension of $\goth g$.
  Let $\goth g$ be the smallest for which the statement is not
  true. Let $L_\mu$ occur in both $\Gamma _0 (G/P^1,  L  _{\lambda}(\goth p^1))$ 
and $\Gamma _0 (G/P^1,  L _{\lambda-\alpha}(\goth p^1))$.
Assume first that $\mu\neq 0$ if
  $\goth g=\goth{osp}(2k+1,2k)$ or $\goth{osp}(2k+2,2k)$. 
Let $\omega$ be
  the highest weight of $E$, ($\omega=\varepsilon_1$ if $\goth
  g=\goth{osp}(2k+1,2k)$ or $\goth{osp}(2k+2,2k)$ and $\delta_1$ if
 $\goth g=\goth{osp}(2k,2k)$), and let $\tau=\chi_{\mu+\omega}$.   
Clearly $\tau$ satisfies the conditions of
  Lemma ~\ref{lp0}, and   $(L_{\mu}\otimes E)^{\tau}=L_{\mu+\omega}$.
Therefore, by Corollary ~\ref{coroblocks}, $L_{\mu+\omega}$ occurs in both 
$\Gamma _0 (G/P^1,  (L_{\lambda}(\goth p^1)\otimes E)^{\Phi^{-1}(\tau)})$ and 
$\Gamma _0 (G/P^1,  (L _{\lambda-\alpha}(\goth p)\otimes E)^{\Phi^{-1}(\tau)})$. 

Note that the first mark of $\mu$ is strictly less than that of 
$\lambda-\alpha$ by Lemma \ref{l2}, hence
one can apply  Lemma ~\ref{lp00} to $\goth s^1$ and get some dominant
  weights $\lambda_1$ and $\lambda_2$ such that
$$(L_{\lambda}(\goth p^1)\otimes E)^{\Phi^{-1}(\tau)}=L_{\lambda_1}(\goth p^1),$$
$$(L_{\lambda-\alpha}(\goth p^1)\otimes
E)^{\Phi^{-1}(\tau)}=L_{\lambda_2}(\goth p^1).$$ Moreover,
it is easy to check that $\lambda_2=\lambda_1-\alpha$. 
Therefore 
$K^{\lambda_i,\mu+\omega}_{G,P^1}(0)\neq 0$ for $i=1,2$.

By Lemma ~\ref{lp0}, 
$L_{\overline{\mu+\omega}}$ occurs in both
$\Gamma _0 (G_{\tau}/P^1_{\tau},  L_{\bar{\lambda}_1}(\goth p^1_{\tau}))$
and 
$\Gamma _0 (G_{\tau}/P^1_{\tau},  (L_{\bar{\lambda}_2}(\goth p^1_{\tau}))$.
Note that $\bar\lambda_2=\bar\lambda_1-\alpha'$, where $\alpha'$ is
the analogue of $\alpha$ for $\goth g_{\tau}$.
Thus, since the statement of the lemma is not true for $\goth g_{\tau}$, that
contradicts the minimality of $\goth g$.

If $\mu=0$ and $\goth  g=\goth{osp}(2k+1,2k)$ or
$\goth{osp}(2k+2,2k)$, we can not apply Lemma ~\ref{lp00}. In the case
$\goth  g=\goth{osp}(2k+1,2k)$  one should use $\mu'=\varepsilon_1$ and 
(\ref{sw1}). If $\goth  g=\goth{osp}(2k+2,2k)$, we use the automorphism
$\sigma$ induced by the symmetry of the Dynkin diagram. Note that $\sigma$
acts trivially on the modules with central character $\chi$ but
switches some simple modules in the block with central
character $\tau$. One gets  
$$(L_{\lambda}(\goth p^1)\otimes
E)^{\Phi^{-1}(\tau)}=L_{\lambda_1}(\goth p^1)\oplus L^\sigma_{\lambda_1}(\goth p^1),$$
$$(L_{\lambda-\alpha}(\goth p^1)\otimes
E)^{\Phi^{-1}(\tau)}=L_{\lambda_2}(\goth p^1)\oplus L^\sigma_{\lambda_2}(\goth p^1)$$
for some dominant weights $\lambda_1$ and $\lambda_2$ such that $\lambda_2=\lambda_1-\alpha$. Since
$L_\omega$ is $\sigma$-invariant, it occurs in the image of the
functor $\Gamma_0$ of all four
summands, and one can finish the proof as in general case.
\end{proof}

\begin{leme}\label{lp4} - One has the exact sequence:
$$ 0 \rightarrow L_{\lambda-\alpha}\rightarrow V_{\lambda}\rightarrow L_{\lambda}\rightarrow 0.$$
\end{leme}
\begin{proof} - Recall that we have the exact sequence 
$$0\rightarrow V_{\lambda}\rightarrow (L_\nu\otimes E)^\chi\rightarrow\Gamma _0 (G/P^1,  L _{\lambda - \alpha}(\goth p ^1)) \rightarrow 0.$$

From Lemma ~\ref{lp1} we know that $L_{\lambda-\alpha}$
  is a submodule of $V_{\lambda}$ 
and $L_{\lambda}$ is the quotient of $V_{\lambda}$ by the unique maximal
submodule $N$. Suppose that
$N$ has another subquotient $L_{\mu}$ with $\mu\neq\lambda-\alpha$. 
Denote by $N'$ the
orthogonal complement to $N$ with respect to the contravariant form on
$(L_{\nu}\otimes E)^{\chi}$. Since $N'$ has a simple subquotient
$L_{\lambda}$, we have $V_\lambda \subset N'$. But $(L_{\mu}\otimes
E)^{\chi}/N'$ is isomorphic to the module contragredient to $N$, hence
it must have a subquotient $L_{\mu}$. Thus, we obtain that $L_{\mu}$
is a subquotient in both $\Gamma _0 (G/P^1,  L _{\lambda}(\goth  p^1))$
and $\Gamma _0 (G/P^1,  L _{\lambda-\alpha}(\goth p^1))$.
That contradicts Lemma ~\ref{lp3}.
\end{proof}

The identity (\ref{r1}) , the exact sequence (\ref{r2}) and Lemma ~\ref{lp4}
imply Proposition ~\ref{pt1}.
\el
\el

\noi {\bf Terminology} - A pair of weights $(\lambda,\mu)$ is called
{\it exceptional}, if
$K^{\lambda,\mu}_{G,P^1}(z)\neq 0$ and the first coordinate of $\mu$
is less than the second coordinate of $\lambda$.

\begin{prop}\label{pt2} - Assume $k\geq 2$, $\lambda$ is such that $a_1 = a_2 +1$,
 and $a_2 \neq \frac{1}{2}, -\frac{1}{2},0$. Assume that
 $(\lambda,\mu)$ is not an exceptional pair. 
One has:

i) $K_{G,P^1} ^{\lambda , \mu}(z) = z K_{P^1,P^2} ^{\lambda-\alpha ,
  \mu}(z)$ if $\mu \neq \lambda, \lambda - \alpha$;

ii) $K_{G,P^1} ^{\lambda , \lambda} (z) = 1$;

iii) $K_{G,P^1} ^{\lambda , \lambda - \alpha} (z) = 0$.

\end{prop}

 Let 
$$\nu=\lambda-\delta_1\,\,\text{for} \,\,\goth g=\goth{osp}(2k+1,2k)\,\, \text{or} \,\,\goth{osp}(2k+2,2k);$$
$$\nu=\lambda-\varepsilon_1\,\,\text{for}\, \goth g=\goth{osp}(2k,2k).$$
Note that $\nu$ is not dominant but is $\goth p^1$-dominant. Observe also
that
$\nu$ satisfies the conditions of Lemma ~\ref{acycl} for $\goth q=\goth p^2$.
Therefore
\begin{equation}\label{acnu}
\Gamma_i(G/P^2,\mathcal L_{\nu}(\goth p^2))=0\, \text{for all} \,i\geq 0.
\end{equation}

\begin{leme}\label{lp5} - One has  
$K^{\nu,\mu}_{G,P^1}(z)=zK^{\nu,\mu}_{P^1,P^2}(z)$ for any $\mu\neq\nu$,
and $K^{\nu,\nu}_{G,P^1}(z)=0$.
\end{leme} 

\begin{proof} - The notations are adapted from those of the proof of
  Theorem ~\ref{ind}. Let $\pi$ denote the canonical projection
  $\pi:G/P^2 \to G/P^1$.
The derived functors $(R^q\pi_*\mathcal {L}_\nu(\goth p^2)^*)$ have the
following property which can be easily obtained from Lemma ~\ref{l1}
applied to $\Gamma_i(P^1/P^2, {L}_\nu(\goth p^2))$: 

 - the second coordinate of $\mu+\rho$ for any $\mu\neq\nu$ such that $L_{\mu}$ occurs in
$(R^q\pi_*\mathcal {L}_\nu(\goth p^2)^*)$ is strictly less than
$a_1-1$, hence any such $\mu$ is $\goth p^1$-typical;

 - $L_{\nu}$ only occurs in $(R^0\pi_*\mathcal {L}_\nu(\goth p^2)^*)$.

Hence the second terms of the Leray spectral sequence are all zero except 
$E_2^{p,0}$ and $E_2^{0,q}$ (for any $p,q$): it is just the typical 
Lemma~\ref{typ}.

Now the identity (\ref{acnu}) asserts that this spectral sequence converges
to zero. So the conclusion is that there is an isomorphism between
$E_2^{p,0}$ and $E_2^{0,p-1}$ and $E_2^{0,0}=0$ (it can't come from
anybody and in the end it should be zero). So
$$H^p(G/P^1,\mathcal {L}_{\nu}(\goth p^1)^*)\simeq
H^0(G/P^1,R^{p-1}\pi_*\mathcal {L}_{\nu}(\goth p^2)^*)$$
and
$$R^{p-1}\pi_*\mathcal {L}_{\nu}(\goth p^2)^*)=H^{p-1}(P^1/P^2,\mathcal {L}_{\nu}(\goth p^2)^*_{|fibre}).$$
We write the decomposition in the Grothendieck group of $\goth p^1$-modules:
$$[H^{p-1}(P^1/P^2,\mathcal {L}_{\nu}(\goth
p^2)^*_{|fibre})]=\sum_{\mu}^{} {}^{p-1}K^{\nu , \mu} _{P^1,P^2}[L_{\mu}(\goth p^1)^*]$$
and all the $\mu$-s are $\goth p^1$-typical except $\mu=\nu$. 
Using the typical lemma ~\ref{typ}, we instantly get that
$$[H^p(G/P^1,\mathcal {L}_{\nu}(\goth
p^1)^*)]=\sum_{\mu\neq \nu}^{}{}^{p-1}K^{\nu , \mu} _{P^1,P^2}[L_{\mu}(\goth p^1)^*]$$
 the equality holding in the
Grothendieck group of $\goth g$-modules. This finishes the proof.
\end{proof}

Now we are ready to prove Proposition ~\ref{pt2}. 
Since $(\lambda,\mu)$ is not an exceptional pair, the first coordinate
of $\mu$ equals either the first coordinate of $\lambda$ or the second
coordinate of $\lambda$. In the former case $\lambda=\mu$ (see
Corollary ~\ref{firstmark}). So we
consider the latter case. Let  
$\omega$ stand for the highest weight of the standard $\goth g$-module.  
By straightforward check of weights we get
$$(L_{\lambda}(\goth p^1)\otimes E)^{\Phi^{-1}(\chi_{\nu})}=L_{\nu}(\goth p^1),$$
$$(L_{\mu}\otimes E)^{\chi_{\nu}}=L_{\mu+\omega}.$$
On the other hand, if $\zeta$ is another weight such that 
$K_{G,P^1}^{\lambda , \zeta} (z)\neq 0$, we have 
$(L_{\zeta}\otimes E)^{\chi_{\nu}}=L_{\zeta+\omega}$ if the first
coordinate of $\zeta$ equals the second coordinate of $\lambda$, and 
$(L_{\zeta}\otimes E)^{\chi_{\nu}}=0$ otherwise. Hence if $\zeta\neq \mu$,
$(L_{\zeta}\otimes E)^{\chi_{\nu}}\neq L_{\mu+\omega}$. 
Therefore, using Corollary ~\ref{coroblocks}  we obtain
$$K_{G,P^1} ^{\lambda , \mu} (z)=K_{G,P^1} ^{\nu, \mu+\omega} (z).$$
By Lemma ~\ref{lp5}
$$K_{G,P^1} ^{\nu, \mu+\omega} (z)=z K_{P^1,P^2} ^{\nu, \mu+\omega} (z).$$
Since $K_{P^1,P^2} ^{\nu, \mu+\omega} (z)$ does not depend on the
first coordinates of the weights, we have
$$K_{P^1,P^2} ^{\nu, \mu+\omega} (z)=K_{P^1,P^2} ^{\nu-\omega, \mu} (z).$$
Note that $\nu-\omega=\lambda-\alpha$, therefore 
(i) of Proposition ~\ref{pt2} is proven. To show (ii) observe that
$L_{\lambda}$ occurs with multiplicity
1 in $\Gamma_0(G/P,L_{\lambda})$  by Lemma ~\ref{l2} and does not
occur in higher cohomology groups by Lemma ~\ref{l1}. Finally, 
$\lambda-\alpha$ is not dominant, hence (iii) is
trivial.

\section{Pretails}
We keep the assumptions of the previous section. Below we list all
weights which do not satisfy the conditions of Proposition ~\ref{pt1}
or Proposition ~\ref{pt2}:

(1) trivial weight $\lambda=0$;

(2) the highest weight $\lambda=\varepsilon_1$ of the standard
representation in the case $\goth g=\goth {osp}(2k+1,2k)$;

(3) $\lambda=\varepsilon_1+\delta_1$ for $\goth g=\goth {osp}(2k,2k)$ or
$\goth{osp}(2k+2,2k)$;

(4) $\lambda=2\varepsilon_1+\delta_1$ for $\goth g=\goth{osp}(2k+1,2k)$ or
$\goth{osp}(2k+1,2k+2)$;

(5) $\lambda=2\varepsilon_1+\varepsilon_2+\delta_1$ for $\goth g=\goth{osp}(2k+1,2k)$ ($k>2$).

In the first case the tail of $\lambda$ has length $k$, hence
$\goth q_{\lambda}=\goth g$. In the other four cases the tail has
length $k-1$, we call such weights {\it pretail} weights. The goal of
this section is to calculate $K_{G,P^1}^{\lambda,\mu}(z)$ for every pretail
weight $\lambda$.
  
Let $\goth q$ be the maximal parabolic subalgebra  corresponding to the first simple
root, $\goth l$ be its Levi part. Clearly, $\goth q\supset \goth p^1$. 

\begin{leme}\label{lt1} - For any pretail weight $\lambda$ we have
$$K^{\lambda,\mu}_{G,P^1}(z)=K^{\lambda,\mu}_{G,Q}(z)$$
for all $\mu$.
\end{leme}
\begin{proof} - Consider the canonical projection $\pi:
  G/P^1\longrightarrow G/Q$.
Note that a pretail weight $\lambda$ is $\goth p^1\cap\goth l$-typical. Hence
$R^q\pi^*\mathcal L_{\lambda}(\goth p^1)^*=0$ for $q>0$ and 
$R^0\pi^*\mathcal L_{\lambda}(\goth p^1)^*=\mathcal L_{\lambda}(\goth
q)^*$. Now the statement follows immediately from Leray spectral sequence.
\end{proof}
 
\begin{leme}\label{triv1} - Let $\goth g=\goth{osp}(2k,2k)$ or $\goth {osp}(2k+2,2k)$.
Then $K^{0,\mu}_{G,Q}(z)=0$ if $\mu\neq 0$,
$$K^{0,0}_{G,Q}(z)=1+z^{2k-1}\,\,\text{  if}\,\, \goth g=\goth{osp}(2k,2k);$$
$$K^{0,0}_{G,Q}(z)=1+z^{2k}\,\,\text{  if}\,\, \goth g=\goth{osp}(2k+2,2k).$$
\end{leme}

\begin{proof} - By Definition ~\ref{penkovrem} and Penkov's remark 
$$Ch(\Gamma_i(G/Q,\mathbb C))\leq
Ch(\Gamma_i(G_0/Q_0,S^{\bullet}(\goth g/(\goth g_0\oplus\goth q_1)))).$$
We are going to describe the simple components of the $\goth q_0$-module 
$S^{\bullet}(\goth g/(\goth g_0\oplus\goth q_1))$. 

If $\goth g=\goth{osp}(2k,2k)$, let $E'$ denote the standard 
$\goth o(2k)$-module and $\goth q''=\goth q\cap \goth {sp}(2k)$, then
$\goth q_0=\goth {o}(2k)\oplus \goth q''$ and one has
the following isomorphism of $\goth q_0$-modules 
$$ S^p(\goth g/(\goth g_0\oplus\goth q_1))\simeq\Lambda^p(E')\boxtimes L_{-p\delta_1}(\goth
q''),$$
where $\boxtimes$ means the tensor product as $\mathbb C$-vector spaces.
  If $\rho''$ is the half-sum of positive roots of $\goth{sp}(2k)$, then
$-p\delta_1+\rho''$  is not regular for all $p$ except $p=0$ or
$2k$. The classical Borel-Weil-Bott theorem shows that there are two
non-zero cohomology groups in degree $0$ and $2k-1$. Since Lemma
~\ref{l2} implies that 
$\Gamma_0(G/Q,\mathbb C)\neq 0$, 
these two components can not cancel in the filtered module. 

If $\goth g=\goth{osp}(2k+2,2k)$, let $E''$ denote the standard 
$\goth{sp}(2k)$-module and $\goth q'=\goth q\cap \goth {0}(2k+2)$, then
$\goth q_0=\goth q'\oplus\goth{sp}(2k)$ and one has the following
isomorphism of $\goth q_0$-modules
$$ S^p(\goth g/(\goth g_0\oplus\goth q_1))\simeq L_{-p\varepsilon_1}(\goth
q')\boxtimes \Lambda^p(E'').$$
Further arguments are exactly the same as in the previous case.  
\end{proof}

\begin{leme}\label{triv2} -  
Let $\goth g=\goth{osp}(2k+1,2k)$.
Then
$$K^{\varepsilon_1,\mu}_{G,Q}(z)=\left\{\begin{array}{l} 0 \; if \;
    \mu\neq \varepsilon_1,0 \\
    1 \; if \; \mu = \varepsilon _1 \\
    z^{2k-1} \; if \; \mu =0
    \end{array}\right ..$$

\end{leme}

\begin{proof} - We do calculations as in the previous lemma.
Let $E''$ denote the standard 
$\goth {sp}(2k)$-module and $\goth q'=\goth q\cap \goth {0}(2k+1)$, then
$\goth q_0=\goth q'\oplus \goth{sp}(2k) $ and 
we have the following isomorphisms of $\goth q_0$-modules
$$L_{\varepsilon_1}(\goth q)\simeq L_{\varepsilon_1}(\goth q'),$$

$$L_{\varepsilon_1}(\goth q)\otimes S^p(\goth g/\goth g_0\oplus\goth
q_1)\simeq L_{(1-p)\varepsilon_1}(\goth q')\boxtimes\Lambda^p(E'').$$

There are exactly three non-acyclic components  
$$L_{\varepsilon_1}(\goth q'),E'',
L_{(1-2k)\varepsilon_1}(\goth q')\boxtimes \Lambda^{2k}(E'').$$
These components give rise to the standard $\goth g$-module in degree
$0$ and  the trivial module in degree $2k-1$.
\end{proof}

\begin{prop}\label{pt3} - Let $\goth g=\goth{osp}(2k,2k)$ or $\goth
  {osp}(2k+2,2k)$ and $\lambda=\varepsilon_1+\delta_1$ be the pretail weight.

i) If $\goth g=\goth{osp}(2k,2k)$ and $k>1$, then
$$K^{\lambda ,\mu}_{G,P^1}(z)=\left\{\begin{array}{l} 1\; if \; \mu = \lambda \;  \\1+z^{2k-2} \; if \; \mu = 0 \\ 0 \; else \end{array}\right ..$$

ii) If $\goth g=\goth{osp}(2,2)$, then
$$K^{\lambda ,\mu}_{G,P^1}(z)=\left\{\begin{array}{l} 1\; if \; \mu =
    \lambda , 0 \;  \\ 0 \; else \end{array}\right ..$$

iii) If $\goth g=\goth{osp}(2k+2,2k)$ then
$$K^{\lambda ,\mu}_{G,P^1}(z)=\left\{\begin{array}{l} 1\; if \; \mu = \lambda \;  \\z^{2k-1} \; if \; \mu = 0 \\
0 \; else \end{array}\right ..$$
\end{prop}
\begin{proof} - We will calculate $K^{\lambda ,\mu}_{G,Q}(z)$ instead of
  $K^{\lambda ,\mu}_{G,P^1}(z)$ (see Lemma ~\ref{lt1}.)
Let $\omega$ be the highest weight of the standard  module $E$. It is
not difficult to see that $\omega$ is $\goth q$-typical. In all cases
except $\goth{osp}(2,2)$ \footnote{for $\goth{osp}(2,2)$ this does
  not work since the standard $\goth {o}(2)$-module is reducible and 
$(L_{\lambda}(\goth q)\otimes E)^{\Phi^{-1}(\chi)}$ has one more
subquotient $L_{-\varepsilon_1+\delta_1}(\goth q)$.}
there is the short exact sequence of $\goth q$-modules:
$$0\rightarrow L_{\lambda}(\goth q)\rightarrow (L_{\omega}(\goth
q)\otimes E)^{\Phi^{-1}(\chi)}\rightarrow L_0(\goth q)=\mathbb
C\rightarrow 0,$$
where $\chi = \chi _{\lambda}$.
As in the proof of Proposition ~\ref{pt1}, the corresponding long exact
sequence degenerates in (\ref{r1}) and (\ref{r2}). It remains to study the
structure of $(L_{\omega}\otimes E)^{\chi}=(E\otimes E)^{\chi}$. It is
different in the two cases.

If $\goth g=\goth{osp}(2k+2,2k)$, $k>1$, then
$$(E\otimes E)^{\chi}=L_{\varepsilon_1+\delta_1}\oplus L_0,$$
and Lemma ~\ref{triv1} implies (iii).

If $\goth g=\goth{osp}(2k,2k)$, then $E\otimes E$ has two trivial
subquotients. (One can see that, for instance, looking at $E\otimes E^*$
for $\goth{gl}(2k,2k)$). Therefore, for a suitable $V_{\lambda}$, one has the exact sequences
$$0\rightarrow V_{\lambda}\rightarrow (E\otimes E)^{\chi}\rightarrow
L_0\rightarrow 0$$
and
$$0\rightarrow L_0\rightarrow V_{\lambda}\rightarrow L_{\lambda}\rightarrow 0.$$
Now (i) follows from Lemma ~\ref{triv1}.

The case $\goth {osp}(2,2)$ can be easily done by a straightforward
calculation similar to those in the two previous lemmas, and we leave it
to the reader.
\end{proof}

\begin{prop}\label{pt4} - Let $\goth g=\goth{osp}(2k+1,2k)$.

i) Let $\lambda_1=2\varepsilon_1+\delta_1$ then
$$K^{\lambda_1,\mu}_{G,P^1}(z)=\left\{\begin{array}{l} 1\; if \; \mu = \lambda_1 \;or\;\varepsilon_1 \; 
    \\z^{2k-2}\; if\;\mu=0\\
0 \; else \end{array}\right ..$$

ii) If $\lambda_2=2\varepsilon_1+\varepsilon_2+\delta_1$ then
$$K^{\lambda_2,\mu}_{G,P^1}(z)=\left\{\begin{array}{l} 1\; if \; \mu = \lambda_2\;or\,0  \\z^{2k-2} \; if \; \mu = \varepsilon_1 \\
0 \; else \end{array}\right ..$$
\end{prop} 
\begin{proof} - 
Let us prove (i). As in the previous proof, using Lemma
~\ref{lt1}, we may calculate $K^{\lambda ,\mu}_{G,Q}(z)$ instead of
  $K^{\lambda ,\mu}_{G,P^1}(z)$. Consider the exact sequence of $\goth  q$-modules
$$0\rightarrow L_{\lambda_1}(\goth q)\rightarrow (L_{\lambda_1-\delta_1}(\goth
q)\otimes E)^{\Phi^{-1}(\chi)}\rightarrow L_{\lambda_1-\alpha}(\goth
q)=L_{\varepsilon_1}(\goth q)\rightarrow 0.$$
We leave to the reader to check that all arguments in the proof of
Proposition ~\ref{pt1} go through and it holds for
$\lambda=\lambda_1$. Thus, (i) follows from Lemma ~\ref{triv2}.

To show (ii) just use the switch functor.
\end{proof}

\section{Exceptional pairs}

The goal of this section is to describe exceptional pairs. It is
convenient to do in terms of weight diagrams. First, we fix some
terminology.
We call $<$ and $>$ 
{\it core symbols}. In what
follows we refer to $0$ (resp. $\frac{1}{2}$) as the {\it tail position}
and denote it by $s_0$.
For any $s<t$ in $\mathbb T$ we  denote by
$l_f(s,t)$ the number of $\times$-s minus the
number of $0$-s strictly between $s$ and $t$. By $|f|$ we denote the
double number of $\times$-s plus the number of core symbols at the
tail position.

First, by the results of the previous section, if $(\lambda,\mu)$ is
exceptional and $\lambda$ is a pretail, then $\goth g=\goth{osp}(2k+1,2k)$, $\lambda=2\varepsilon_1+\varepsilon_2+\delta_1$ and $\mu=0$.

\begin{prop}\label{exc} - Assume that $\lambda$ is not a pretail weight and let
  $t+1$ be the position of the rightmost $\times$ in $f_{\lambda}$. The pair $(\lambda,\mu)$ is exceptional if and only if the
  following conditions are true

(a) $f_{\lambda}$ is obtained from $f_{\mu}$ by moving two $\times$
from the tail position to the adjacent non-tail positions $t$
and $t+1$ (if $f_{\lambda}$ and $f_{\mu}$ have signs, they remain the same);

(b) $l_{f_\lambda}(s_0,t)$ is odd for 
$\goth g=\goth {osp}(2k+1,2k)$ or $\goth {osp}(2k,2k)$ and 
even for $\goth g=\goth {osp}(2k+2,2k)$;

(c) $l_{f_\lambda}(s,t)\leq 0$ for any $s<t$;

(d) $l_{f_\lambda}(s_0,t)+|f_{\lambda}|>0 $.

If $(\lambda,\mu)$ is exceptional, then $K^{\lambda,\mu}_{G,P^1}(z)=1$. 

\end{prop}

The proof of Proposition ~\ref{exc} takes the rest of the section.

\begin{leme}\label{degten} - 

i) Let $\goth g=\goth{osp}(2k,2k)$, $\tau=\chi_{\delta_1}$. If 
$a_i=1$ for some $i$, then $(L_{\lambda}\otimes E)^{\tau}=0$.

ii)  Let $\goth g=\goth{osp}(2k+2,2k)$, $\tau=\chi_{\varepsilon_1}$,
$\kappa$ be a dominant weight with central character $\tau$. Then
$(L_{\kappa}\otimes E)^{\chi}=L_{\kappa-\varepsilon_i}\oplus L_{\kappa+\delta_i}$, where 
$i$ is such that $a_i=1$.

iii) Let $\goth g=\goth{osp}(2k+1,2k)$, $\tau=\chi_{2\varepsilon_1}$. If 
$a_i=\frac{3}{2}$ for some $i$, then $(L_{\lambda}\otimes E)^{\tau}=0$.
\end{leme}

\begin{proof} - (i) Assume first that
  $\lambda=\varepsilon_1+\delta_1$. If $(L_{\lambda}\otimes  E)^{\tau}\neq 0$, then 
$(L_{\lambda}\otimes E)^{\tau}=E$ (by looking at the weights). But as
follows from the calculations in the proof of Proposition ~\ref{pt3} 
$$\operatorname{Hom}_{\goth g}(L_{\lambda},E\otimes E)=\operatorname{Hom}_{\goth g}(L_{\lambda}\otimes E,E)=0.$$
Hence $(L_{\lambda}\otimes E)^{\tau}=0$.

In the general case, by above, we have $(L_{\lambda}(\goth p^{i-1})\otimes E)^{\Phi^{-1}(\tau)}=0$. 
Hence $$(\Gamma_0(G/P^{i-1},L_{\lambda}(\goth p^{i-1})) \otimes
E)^\tau=0.$$
Since $L_{\lambda}$ is a quotient of $\Gamma_0(G/P^{i-1},L_{\lambda}(\goth p^{i-1}))$ 
we obtain $(L_{\lambda}\otimes E)^{\tau}=0$.

(ii) We use $(E\otimes E)^{\chi}=L_{\varepsilon_1+\delta_1}\oplus L_0$. 
That implies 
$$(L_{\kappa}(\goth p^{i-1})\otimes E)^{\Phi^{-1}(\chi)}=L_{\kappa-\varepsilon_i}(\goth p^{i-1})\oplus L_{\kappa+\delta_i}(\goth p^{i-1}).$$
(Keep in mind that the standard module for the Levi part of $\goth p
^{i-1}$ has highest weight $\varepsilon _i$).

Therefore
$$(\Gamma_0(G/P^{i-1},L_{\kappa}(\goth p^{i-1}))\otimes
E)^{\chi}=\Gamma_0(G/P^{i-1},L_{\kappa-\varepsilon_i}(\goth
p^{i-1}))\oplus \Gamma_0(G/P^{i-1},L_{\kappa+\delta_i}(\goth p^{i-1})).$$
Clearly, $(L_{\kappa}\otimes E)^{\chi}$ is a quotient of the right hand
side. 
Let $S$ denote the submodule of $\Gamma_0(G/P^{i-1},L_{\kappa}(\goth
p^{i-1}))$ such that $L_{\kappa}=\Gamma_0(G/P^{i-1},L_{\kappa}(\goth
p^{i-1}))/S$. All weights of $S$ are less than $\kappa$. Therefore
$(S\otimes E)^{\chi}$ can not have simple components isomorphic to  
$L_{\kappa-\varepsilon_i}$ or $L_{\kappa+\delta_i}$, hence
$(L_{\kappa}\otimes E)^{\chi}$ must have these components. Since 
$(L_{\kappa}\otimes E)^{\chi}$ is contragredient and its $\goth
b$-singular vectors may have only weights 
 $\kappa-\varepsilon_i$ and 
$\kappa+\delta_i$, we have 
$(L_{\kappa}\otimes E)^{\chi}=L_{\kappa-\varepsilon_i}\oplus
L_{\kappa+\delta_i}$
(see the proof of Lemma ~\ref{lp2} for details).

(iii) First, observe that if $a_{i+1}=\frac{1}{2}$ then there are no
dominant weights of the form $\lambda+\gamma$ with $\gamma$ being a
weight of $E$ such that $\chi_{\lambda+\gamma}=\tau$. Hence we may
assume that  $a_{i+1}=-\frac{1}{2}$ or $i=k$. By comparison of weights
either $(L_\lambda\otimes E)^\tau=0$ or $(L_\lambda\otimes
E)^\tau=L_{\lambda-\delta_i}$. Consider the exact
sequence as in the proof of Lemma ~\ref{lp00}
$$0\rightarrow S\rightarrow M_{\lambda-\beta}\rightarrow
M_{\lambda}\rightarrow F\rightarrow 0,$$
where $\beta=\varepsilon_i+\delta_i$. Check that $(M_{\lambda}\otimes E)^{\chi}$  
has a filtration by Verma modules such that their highest weights are
not greater than $\lambda-\delta_i$. Hence $L_{\lambda-\delta_i}$ occurs in  $(M_{\lambda}\otimes E)^{\chi}$
with multiplicity 1. Now one can finish the proof by the same
arguments as in the proof of Lemma ~\ref{lp00}. 
\end{proof}

Till the end of this section $\beta=\varepsilon_2+\delta_2$, $\omega$,
as usual, denotes the highest weight of $E$.

\begin{leme}\label{ten1} - Let $\mu$ and $\nu$ be dominant weights
  with  trivial central character,  $\tau=\chi_{\mu+\omega}$. Assume also
  that $\mu\neq 0$ in the case $\goth g=\goth{osp}(2k+1,2k)$ or $\goth{osp}(2k+2,2k)$. If
  $(L_\nu\otimes E)^{\tau}=L_{\mu+\omega}$, then $\nu=\mu$.
\end{leme}

\begin{proof} - For $\mu\neq \varepsilon_1,0$ the lemma easily follows from Lemma ~\ref{lp00}. The
  condition $\mu\neq 0,\varepsilon_1$ ensures that
  $\mu+\omega=\nu+\varepsilon_i$ or $\mu+\omega=\nu+\delta_i$ for some
  $i$. Since $\chi_{\mu}=\chi_{\nu}=\chi_0$, we have $\mu=\nu$.

If $\mu=\varepsilon_1$, then $\goth g=\goth{osp}(2k+1,2k)$ and by
Lemma ~\ref{degten}(iii) we again have $\mu+\omega=\nu+\varepsilon_i$ for some
  $i$ and the statement follows by the same reason. 

Similarly, if $\goth g=\goth{osp}(2k,2k)$ and $\mu=0$, one can prove
the statement using Lemma ~\ref{degten}(i).

\end{proof}

\begin{leme}\label{deg2} - Let $(\lambda,\mu)$ be an exceptional pair,
  $\mu_1=\mu+\omega$ and $\tau=\chi_{\mu+\omega}$. Assume again 
  that $\mu\neq 0$ in the case $\goth g=\goth{osp}(2k+1,2k)$ or 
$\goth{osp}(2k+2,2k)$.
 There exists a dominant weight $\lambda_1$
  such that 
$$(L_{\lambda}(\goth p^1)\otimes E)^{\Phi^{-1}(\tau)}=L_{\lambda_1}(\goth p^1)$$
and
$$K^{\lambda,\mu}_{G,P^1}(z)=K^{\lambda_1,\mu_1}_{G,P^1}(z)=K^{\bar{\lambda}_1,\bar{\mu}_1}_{G_\tau,P_\tau^1}(z).$$

\end{leme}

\begin{proof}  - Since $(L_{\mu}\otimes E)^{\tau}=L_{\mu_1}$, we know that
  $L_{\mu_1}$ occurs in some $$\Gamma_i(G/P^1,(L_{\lambda}(\goth  p^1)\otimes E)^{\Phi^{-1}(\tau)}).$$ Hence 
$(L_{\lambda}(\goth p^1)\otimes E)^{\Phi^{-1}(\tau)}\neq 0$. If $E'$
denotes the standard $\goth p^1$-module we have the following identity
in the Grothendieck group:
$$[(L_{\lambda}(\goth p^1)\otimes E)^{\Phi^{-1}(\tau)}]=[L_{\lambda+\varepsilon_1}(\goth p^1)]+[L_{\lambda-\varepsilon_1}(\goth p^1)]+$$
$$+[L_{\lambda+\delta_1}(\goth p^1)]+[L_{\lambda-\delta_1}(\goth p^1)]+[(L_{\lambda}(\goth p^1)\otimes E')^{\Phi^{-1}(\tau)}].$$
Since the first coordinate of $\mu$ is less than the second coordinate
of $\lambda$,
$$(L_{\lambda\pm\varepsilon_1}(\goth p^1))^{\Phi^{-1}(\tau)}=(L_{\lambda\pm\delta_1}(\goth p^1))^{\Phi^{-1}(\tau)}=0,$$
and
$$(L_{\lambda}(\goth p^1)\otimes E)^{\Phi^{-1}(\tau)}=(L_{\lambda}(\goth p^1)\otimes E')^{\Phi^{-1}(\tau)}.$$
Lemma ~\ref{lp00} applied to $\goth p^1$-modules implies that
$(L_{\lambda}(\goth p^1)\otimes E')^{\Phi^{-1}(\tau)}$ is simple,
hence it is isomorphic to $L_{\lambda_1}(\goth p^1)$ for some $\goth p^1$-dominant
$\lambda_1$. By the condition on the first coordinate of $\mu$, $\lambda_1$
is dominant.

Finally, since $(L_{\mu}\otimes E)^{\tau}=L_{\mu_1}$,  the multiplicity
of $L_{\mu_1}$ in $\Gamma_i(G/P^1,(L_{\lambda_1}(\goth  p^1))$ is at
least the same as the multiplicity of $L_{\mu}$ in $\Gamma_i(G/P^1,(L_{\lambda}(\goth  p^1))$.
Hence
${}^iK^{\lambda,\mu}_{G,P^1}\leq{}^iK^{\lambda_1,\mu_1}_{G,P^1}$. But
by Lemma ~\ref{ten1} $(L_{\nu}\otimes E)^{\tau}\neq L_{\mu_1}$ for any
$\nu\neq \mu$ and we must have
${}^iK^{\lambda,\mu}_{G,P^1}={}^iK^{\lambda_1,\mu_1}_{G,P^1}$. Lemma ~\ref{lp0} implies 
$K^{\lambda_1,\mu_1}_{G,P^1}(z)=K^{\bar{\lambda}_1,\bar{\mu}_1}_{G_\tau,P_\tau^1}(z)$.

\end{proof}

\begin{coro}\label{excind} - Use the notations of the previous lemma. 
Let $(\lambda,\mu)$ be an exceptional pair,
$r$ be the position of the rightmost $\times$ in $f_\mu$, and assume
that $r\neq s_0$. Then $f_{\lambda}(r)=\times$,  $f_{\lambda}(r+1)=0$, 
$(\bar{\lambda}_1,\bar{\mu}_1)$ is an
exceptional pair for $\goth g_\tau$ and 
$K^{\lambda,\mu}_{G,P^1}(z)=K^{\bar{\lambda}_1,\bar{\mu}_1}_{G_\tau,P_\tau^1}(z).$ 
Furthermore, $f_{\bar{\mu}_1}$ is obtained form $f_{\mu}$ by removing the
$\times$ from the $r$-th position and  
$$f_{\bar{\lambda}_1}(t)=\left\{\begin{array}{l} f_{\lambda}(t)\; if
    \; t<r \;  \\f_{\lambda}(t+2) \; if \; t\geq r\end{array}\right ..$$
\end{coro}

\begin{leme}\label{deg1} - Let $\goth g=\goth{osp}(2k+2,2k)$, and
  $\lambda=(2,1,0,...|2,1,0,...)$. Then $K^{\lambda,0}_{G,P^1}(z)=1$. 
\end{leme}

\begin{proof} - Let $\kappa=\lambda-\delta_2$ and
  $\tau=\chi_{\kappa}$. By Lemma ~\ref{lp0},  
  $K^{\kappa,\mu}_{G,P^1}(z)=K^{\bar\kappa,\bar\mu}_{G_\tau,P^1_\tau}(z)$. 
Therefore, by  Proposition ~\ref{pt3}, 
$$K^{\kappa ,\mu}_{G,P^1}(z)=\left\{\begin{array}{l} 1\; if \; \mu =    \kappa \;  \\1+z^{2k-2} \; if \; \mu = \varepsilon_1 \\ 0 \; else    \end{array}\right ..$$
By Lemma ~\ref{degten} (ii)
$$(L_{\kappa}(\goth p^1)\otimes E)^{\Phi^{-1}(\chi)}=L_{\lambda}(\goth
p^1)\oplus L_{\lambda-\beta}(\goth p^1).$$
By Propositon ~\ref{pt1} and Proposition ~\ref{pt3}, we obtain
$$K^{\lambda-\beta ,\mu}_{G,P^1}(z)=\left\{\begin{array}{l} 1\; if \; \mu   =  \lambda-\beta \;  \\z^{2k-2} \; if \; \mu = 0 \\ 0 \; else \end{array}\right ..$$
By Corollary ~\ref{coroblocks}
$$(\Gamma_{i}(G/P^1,L_{\kappa}(\goth p^1))\otimes E)^{\chi}=\Gamma_{i}(G/P^1,L_{\lambda}(\goth p^1))\oplus\Gamma_{i}(G/P^1,L_{\lambda-\beta}(\goth p^1)),$$
and by Lemma ~\ref{degten} (ii)
$$(E\otimes E)^{\chi}=L_0\oplus L_{\varepsilon_1+\delta_1},\; (L_{\kappa}\otimes E)^{\chi}=L_{\lambda}\oplus L_{\lambda-\beta}.$$
Therefore
$$K^{\kappa,\varepsilon_1}_{G,P^1}(z)=K^{\lambda-\beta ,0}_{G,P^1}(z)+K^{\lambda,0}_{G,P^1}(z),$$
and we have $K^{\lambda,0}_{G,P^1}(z)=1$.
\end{proof}

\begin{leme}\label{ten2} - Let $\goth g=\goth{osp}(2k+2,2k)$,
  $\tau=\chi_{\omega}$, $\mu$ be a dominant weight with trivial central
  character. One has$ (L_{\mu}\otimes E)^{\tau}=E$ if and only if $\mu=0$ or 
$\varepsilon_1+\delta_1$. 
\end{leme}
\begin{proof} - Since $\omega=\mu\pm\varepsilon_i$ or
  $\mu\pm\delta_i$, the only possible values for $\mu$ are $0$ and
  $\varepsilon_1+\delta_1$. Obviously $(L_{0}\otimes E)^{\tau}=E$. To
  check that $(L_{\varepsilon_1+\delta_1}\otimes E)^{\tau}=E$, use
$$\operatorname{Hom}_{\goth g} (L_{\varepsilon_1+\delta_1}\otimes
E,E)=\operatorname{Hom}_{\goth g} (L_{\varepsilon_1+\delta_1},E\otimes
E)=\mathbb C.$$
\end{proof}

\begin{leme}\label{deg3} - Let $\goth g=\goth{osp}(2k+2,2k)$ (resp. 
$\goth{osp}(2k,2k)$). Then
  $(\lambda,0)$ is an exceptional pair if and only if
$$\lambda=(a,a-1,0,...|a,a-1,0,...),$$
$a\leq 2k-2$ is even (resp. $a\leq 2k-3$ is odd).
If $\lambda$ satisfies the above conditions, then
$K^{\lambda,0}_{G,P^1}(z)=1$.
\end{leme}

\begin{proof} - We will prove first that if $\lambda$ satisfies the
  conditions of Lemma, then $K^{\lambda,0}_{G,P^1}(z)=1$. 
The proof will be done by induction
on $a$ and $k$, the base case $a=2$ is done in Lemma ~\ref{deg1}. In this
proof $\tau=\chi_{\omega}$.

First, let $\goth g=\goth{osp}(2k+2,2k)$. Set
$\lambda_1=\lambda+\varepsilon_3$. By Lemma ~\ref{lp0} and induction
assumption, we have
$$K^{\lambda_1,\varepsilon_1}_{G,P^1}(z)=K^{\bar{\lambda}_1,0}_{G_\tau,P^1_\tau}(z)=1.$$
By Lemma ~\ref{degten}(ii) we have
$$(L_{\lambda_1}(\goth p^1)\otimes
E)^{\Phi^{-1}(\chi)}=L_{\lambda}(\goth p^1)
\oplus L_{\lambda+\varepsilon_3+\delta_3}(\goth p^1),$$
and we have
$$\operatorname{Hom}_{\goth p^1}(L_{\lambda_1}(\goth p^1)\otimes E,L_{\lambda}(\goth p^1))=\operatorname{Hom}_{\goth p^1}(L_{\lambda_1}(\goth p^1),L_{\lambda}(\goth p^1)\otimes E)=\mathbb C.$$
That implies
$$(L_{\lambda}(\goth p^1)\otimes E)^{\Phi^{-1}(\tau)}=L_{\lambda_1}(\goth p^1).$$
Thus, by Corollary ~\ref{coroblocks} and Lemma ~\ref{ten2} we have either 
$K^{\lambda,\varepsilon_1+\delta_1}_{G,P^1}(z)=1$ or 
$K^{\lambda,0}_{G,P^1}(z)=1$. However, the case  
$K^{\lambda,\varepsilon_1+\delta_1}_{G,P^1}(z)=1$ is impossible
by Corollary ~\ref{excind} and we have $K^{\lambda,0}_{G,P^1}(z)=1$.

Now let $\goth g=\goth{osp}(2k,2k)$. Set $\lambda_1=\lambda+\delta_3$.
Using the argument similar to above we have 
$$K^{\lambda_1,\delta_1}_{G,P^1}(z)=1.$$
By  Lemma ~\ref{degten}(i), one has
$$(L_{\lambda}(\goth p^1)\otimes
E)^{\Phi^{-1}(\tau)}=L_{\lambda_1}(\goth p^1).$$
and by  Lemma ~\ref{ten1} $K^{\lambda,0}_{G,P^1}(z)=1$.

Now we will show that if $(\lambda,0)$ is an exceptional pair, then
$\lambda$ must satisfy the conditions of the lemma. The proof again is by
induction on $k$. Consider two cases $f_{\lambda}(1)=\times$ and 
$f_{\lambda}(1)=0$.

In the first case Lemma ~\ref{deg2} implies that 
$\goth g=\goth{osp}(2k+2,2k)$. Let $i$ be such that $a_i=1$ and
$\lambda_1=\lambda-\delta_i$. $K^{\lambda,0}_{G,P^1}(z)\neq 0$ implies
$$(L_{\lambda}(\goth p^1)\otimes
E)^{\Phi^{-1}(\tau)}=L_{\lambda_1}(\goth p^1)$$
and 
$$K^{\lambda_1,\varepsilon_1}_{G,P^1}(z)=K^{\bar{\lambda}_1,0}_{G_\tau,P^1_\tau}(z)\neq 0.$$ 
By induction assumption 
$$\bar{\lambda}_1=(a,a-1,0,...|a,a-1,0,...),
{\lambda_1}=(a+1,a,1,...|a+1,a,0,...),$$
and $K^{\lambda_1,\varepsilon_1}_{G,P^1}(z)=1$.
But 
$$(L_{\lambda_1}(\goth p^1)\otimes
E)^{\Phi^{-1}(\tau)}=L_{\lambda_1-\varepsilon_3}(\goth p^1)
\oplus L_{\lambda_1+\delta_3}(\goth p^1).$$
Since $L_0$ occurs with multiplicity 1 in
$\Gamma_0(G/P^1,L_{\lambda_1-\varepsilon_3}(\goth p^1)\oplus
L_{\lambda_1+\delta_3}(\goth p^1))$, and since we have proved already that
$(\lambda_1-\varepsilon_3,0)$ is exceptional, we obtain that 
$(\lambda_1+\delta_3,0)$
is not exceptional.

Finally, if $f_{\lambda}(1)=0$, take $\lambda_1=\lambda+\varepsilon_i$
(resp. $\lambda_1=\lambda+\delta_i$). Repeat the above arguments to
show that $(\bar{\lambda}_1,0)$ is again an exceptional pair. By
induction assumption
$$\bar{\lambda}_1=(a,a-1,0,...|a,a-1,0,...),$$
hence
$$\lambda=(a+1,a,0,...|a+1,a,0,...).$$
\end{proof}

\begin{leme}\label{deg4} - Let $\goth g=\goth{osp}(2k+1,2k)$. Then
  $(\lambda,0)$ (resp. $(\lambda,\varepsilon_1)$) is an exceptional pair if and only if
$$\lambda+\rho=(a+\frac{1}{2},a-\frac{1}{2},-\frac{1}{2},...|
a+\frac{1}{2},a-\frac{1}{2},\frac{1}{2},...),$$
(resp.
$$\lambda+\rho=(a+\frac{1}{2},a-\frac{1}{2},\frac{1}{2},-\frac{1}{2},...|
a+\frac{1}{2},a-\frac{1}{2},\frac{1}{2},...),)$$
where $a$ is odd and $a\leq 2k-4$.
If $\lambda$ satisfies the above conditions, then
$K^{\lambda,0}_{G,P^1}(z)=1$ (resp. $K^{\lambda,\varepsilon_1}_{G,P^1}(z)=1$).
\end{leme}

\begin{proof} -  It suffices to prove lemma for the pair
$(\lambda,\varepsilon_1)$, since the second case follows by the use of the switch
functor. The proof is similar to one for the previous lemma, but
much easier. Let $\tau=\chi_{2\varepsilon_1}$.
Note that $K^{\lambda,\varepsilon_1}_{G,P^1}(z)\neq 0$ implies  
$$(L_{\lambda}(\goth p^1)\otimes E)^{\Phi^{-1}(\tau)}\neq 0.$$
Then by Lemma ~\ref{degten} (iii) $f_\lambda(\frac{3}{2})=0$,
 $f_\lambda$ has positive sign and 
$$(L_{\lambda}(\goth p^1)\otimes
E)^{\Phi^{-1}(\tau)}=L_{\lambda_1}(\goth p^1),$$
where $\lambda_1=\lambda+\varepsilon_i$ for $i$ such that $a_i=\frac{1}{2}$.
Now use
$$K^{\lambda,\varepsilon_1}_{G,P^1}(z)=K^{\lambda_1,2\varepsilon_1}_{G,P^1}(z) = K^{\bar{\lambda}_1,0}_{G_\tau,P^1_\tau}(z)=K^{\bar{\lambda}'_1,\varepsilon_1}_{G_\tau,P^1_\tau}(z),$$
and proceed by induction decreasing $k$. The process stops when
$\bar{\lambda}_1=2\varepsilon_1+\varepsilon_2+\delta_1$ is a pretail
weight (see Proposition ~\ref{pt4}).
\end{proof}

Now  Proposition ~\ref{exc} follows easily from Corollary ~\ref{excind},
Lemma  ~\ref{deg3} and Lemma  ~\ref{deg4} by induction on number of
$\times$ outside the tail position. 

\section{Combinatorial algorithm for calculation of 
  $K^{\lambda,\mu}_{G,P^1}(z)$ and $K^{\lambda,\mu}_{G,Q_{\lambda}}(-1)$}

The recursion formulae obtained in two previous sections determine uniquely
the Poincar\'e polynomials $K^{\lambda,\mu}_{G,P^1}(z)$ in the most atypical block. In this section
we obtain a closed formula for them and use it to compute $K^{\lambda,\mu}_{G,Q_{\lambda}}(-1)$.
The results will be formulated in terms of weight diagrams. The
recursion uses double induction on the position of the rightmost
$\times$ (Proposition ~\ref{pt1}) and on the total number of $\times$.

We still assume that we are in the most atypical block containing the
trivial module. However, keeping track of the signs of weight
diagrams in the case $\goth g=\goth{osp}(2k+1,2k)$ is annoying. To get rid
off those signs, we use the equivalent block of $\goth{osp}(2k+1,2k+2)$
with trivial central character. Weight diagrams lying in this block
have $<$ at the tail position. (To obtain a signed
diagram from the most atypical block of $\goth{osp}(2k+1,2k)$ from such a diagram $f$, one has
to shift all entries of $f$ except one at the tail position one
position left and assign $-$ if $f(\frac{3}{2})=0$ and $+$ if
$f(\frac{3}{2})=\times$.) 

 By our assumption the core
symbols can only be at the tail position.

If $f(t)$ has $\times$, then let $f_t$ denote the diagram
$g$ obtained from $f$ by removing $\times$ from position $t$,
(naturally  $g(t)=0$ if $t$ becomes empty). Similarly, 
we denote by $f^t$ the diagram obtained from $f$ by adding $\times$ to
position $t$. Thus, $f_s^t$ denote the diagram obtained from $f$ by
moving $\times$ from  $s$ to $t$.

\begin{defi}\label{legalmove} - We say that $g$ is obtained from $f$ by
  a legal move if $s<t$, $f(s)$ contains $\times$, $f(t)=0$, 
$g=f_s^t$  and one of the following conditions holds

(1) $f(s)$ does not have a core symbol, $l_f(s,r)>0$ for any $s<r<t$
and $l_f(s,t)\geq 0$;

(2) $s$ is the tail position, $|g|+l_f(s,r)>0$ for any $s<r<t$ and
$|g|+l_f(s,t)\geq 0$.

The positions $s$ and $t$ are called the start and the end of a legal move $M$. The
degree $l(M)$ of a legal move is defined to be $l_f(s,t)$ in the first
case and $|g|+l_f(s,t)$ in the second case.
A legal move is called ordinary if its start is not a tail. Otherwise
we call a legal move a tail legal move. 
\end{defi}
{\bf Warning}. A legal move is actually the following data: two
diagrams $f$ and $g$ and the degree of the corresponding move. 
Sometimes there are two legal tail moves of different degrees
which transform $f$ to the same diagram $g$. They should be considered
as different moves. For example, the diagram
${}^{\times}_{\times},0,...$ can be transformed to $\times,\times,...$
by two different legal moves, one of degree $0$ and one of degree $2$.

\begin{defi}\label{expmove} - Let $s_0$ denote the tail position. 
We say that $g$ is obtained from
$f$ by an exceptional move if 

$g=(f_{s_0}^s)_{s_0}^{t}$ for some $s<t$, $f(s)=f(t)=0$; 

$l_{f}(a,s)\leq 0$ for all $a<s$ and 
$|g|+l_{f}(s_0,s)$ is a positive odd number;

$l_{f}(s,b)> 0$ for all $s<b<t$ and 
$l_{f}(s,t)\geq 0$.

Define the end 
position of an exceptional move to be
$t$ and the degree to be $l_{f}(s,t)$.
\end{defi}

\begin{prop}\label{legmovprop} - 
Let $\lambda$ and $\mu$ be two dominant weights with trivial central
character and $\lambda\neq 0$. Then ${}^iK^{\lambda,\mu}_{G,P^1}=1$ if
$\mu=\lambda$ and $i=0$ or $f_\lambda$ is
obtained from $f_\mu$ by a legal move (or by an exceptional move)  of degree
$i$, with end at the position of the rightmost $\times$ of $f_\lambda$.
Otherwise, ${}^iK^{\lambda,\mu}_{G,P^1}=0$.
\end{prop}

\begin{proof} - If $\lambda=\mu$ the statement is obvious, if $(\lambda,\mu)$
is an exceptional pair, the statement is proven in
Proposition~\ref{exc}. For other cases the statement follows from
Propositions ~\ref{pt1},~\ref{pt2} by induction on the position of
the rightmost $\times$ and the number of $\times$-s in $f_{\lambda}$,
the base of induction being done in Propositions  ~\ref{pt3},~\ref{pt4} and
Lemma ~\ref{triv2}. One has to translate  Proposition ~\ref{pt3} and
Lemma ~\ref{triv2} from signed diagrams to diagrams with $<$ in the
tail position, which can be done by direct comparison. 
\end{proof}

\begin{coro}\label{legmovcor1} - The multiplicity of any simple module
  in $\Gamma_i(G/P^1, L_{\lambda}(\goth p^1))$ is at most one.
\end{coro}

\begin{coro}\label{legmovcor2} - Let $\lambda$ and $\mu$ be two dominant weights with trivial central
character and $\lambda\neq 0$. Let $t$ be the position of $j$-th
$\times$ in $f_{\lambda}$ counting from the right. Then ${}^iK^{\lambda,\mu}_{P^{j-1},P^j}=1$ if
$\mu=\lambda$ and $i=0$, or if $f_\lambda$ is
obtained from $f_\mu$ by a legal move (or by an exceptional move)  of degree
$i$ with end $t$.
Otherwise, ${}^iK^{\lambda,\mu}_{P^{j-1},P^j}=0$.
\end{coro}

Let $D_{\goth g}$ denote the oriented graph whose vertices are dominant weights
of $\goth g$, and edges are defined as follows:

if $f_{\lambda}$ is obtained from $f_{\mu}$ by a 
legal  move  or
exceptional move, we join $\lambda$ and
$\mu$ by an edge  $\mu\longrightarrow\lambda$. 

We put a {\it label} $(s,t;w)$ on an edge, where $s$ and $t$ are the
start and the end of the corresponding legal move and $w$ is its degree. 
If the move is exceptional we put the label $(s_0:s,t;w)$.

A path consisting of edges corresponding to legal moves with ends 
$t_1,...,t_q$ is called   
{\it decreasing} resp. {\it increasing} if $t_1>...>t_q$
(resp. $t_1<...< t_q$). (It follows immediately from the definition that,
in any path, $t_i\neq t_{i+1}$.) The degree $l(R)$ of a path $R$ is the sum
of the degrees of all legal moves corresponding to the edges included in
$R$. It is straightforward that $D_{\goth g}$ does not have oriented
cycles. 

\begin{theo}\label{combth} - Let $\mathcal P^{>}(\mu,\lambda)$  denote the set 
of all
  decreasing  paths from $\mu$ to $\lambda$. Then
\begin{equation}\label{f2}
K^{\lambda,\mu}_{G,Q_{\lambda}}(-1)=\sum_{R\in \mathcal P^{>}(\mu,\lambda)}
(-1)^{l(R)}.
\end{equation}
\end{theo}
\begin{proof} -  
Let $r$ be the number of $\times$ outside the tail position in $f_\lambda$. 
Then Theorem ~\ref{ind} and Corollary
~\ref{legmovcor2} imply 
\begin{equation}\label{aux}
K^{\lambda,\mu}_{G,Q_{\lambda}}(-1)=\sum_{\mu\leq\mu_1\leq...\leq\mu_r\leq
\lambda}
K^{\lambda,\mu_r}_{P^{r-1},P^r}(-1)K^{\mu_{r},\mu_{r-1}}_{P^{r-1},
P^{r-2}}(-1)...K^{\mu_1,\mu}_{G,P^1}(-1)
=\sum_{R\in {\mathcal P}^{>}(\mu,\lambda)}(-1)^{l(R)}.
\end{equation}

\end{proof}

\begin{rem} - One can easily generalize Theorem ~\ref{combth} to an
  arbitrary block due to Corollary ~\ref{cor22}. Legal and exceptional moves are described
in the same way. We completely ignore core symbols outside the tail position.
\end{rem}

\section{Characters}
In this section we give a combinatorial algorithm for computing 
characters of simple modules. 

Let $\mathcal E_{\lambda}$ denote the
right hand side of formula (\ref{char1}) in Proposition ~\ref{char}
with $\goth p=\goth q_{\lambda}$. Note that since 
$L_{\lambda}(\goth q_{\lambda})$ is one-dimensional,  
$Ch(L_{\lambda}(\goth q_{\lambda}))=e^{\lambda}$. 

The identity (\ref{char1}) provides a linear system of equations, which
can be solved for $Ch (L_{\lambda})$. 
Let $\mathbb K$ denote the infinite matrix with
coefficients 
$K^{\lambda,\mu}_{G,Q_{\lambda}}(-1)$. Then $\mathbb K$ is lower
triangular with $1$ on diagonal. Let $\mathbb D=\mathbb K^{-1}$, and
$D^{\lambda,\mu}$ denote the matrix coefficients of $\mathbb D$.
Then (\ref{char1}) implies
\begin{equation}\label{char2}
Ch (L_{\lambda})=\sum D^{\lambda,\mu} \mathcal E_{\mu}.
\end{equation}

\begin{rem} - One can see that the graphs $D_{\goth g}$ and the
  matrices $\mathbb K$ and $\mathbb D$ for $\goth g=\goth{osp}(2k+2,2k)$ and
  $\goth g=\goth{osp}(2k+1,2k)$ are the same if one identifies 
  weight diagrams by switching $>$ to $<$ at the tail position. It is
  natural to conjecture that the maximal atypical blocks in these two
  cases are equivalent.
\end{rem}

\begin{theo}\label{thch} - Let $\mathcal P^{<}(\mu,\lambda)$ denote the set of 
increasing paths from $\mu$ to $\lambda$ in $D_{\goth g}$, and $|R|$ denote the
number of edges in a path $R$. Then
\begin{equation}\label{f3}
D^{\lambda,\mu}=\sum_{R\in \mathcal P^{<}(\mu,\lambda)}(-1)^{l(R)+|R|}.
\end{equation}
\end{theo}

\begin{proof} - Write $\mathbb K=1+\mathbb U$, where $\mathbb U$
  is strictly low triangular. Then
$$\mathbb D= 1-\mathbb U+\mathbb U^2-...$$
Let $R_1\circ R_2$ denote the concatenation of paths $R_1$ and
$R_2$. Then (\ref{f2}) implies
\begin{equation}\label{f4}
D^{\lambda,\mu}=\sum_{R\in \mathcal P(\mu,\lambda)}\sum_{R=R_1\circ
  R_2...\circ R_i | R_j\in\mathcal P^{>}}(-1)^{i+l(R)},
\end{equation}
where $\mathcal P(\mu,\lambda)$ denotes the set of all paths from $\mu$
to $\lambda$ and $\mathcal P^{>}$ is the set of all decreasing paths.

Any path $R\notin \mathcal P^{<}(\mu,\lambda)$ has more
than one term in the second sum, since there are several ways to write it
as a concatenation of decreasing paths. It is a simple exercise to check
that in this case the second sum is zero. Hence the only paths
contributing to the formula are increasing. Hence (\ref{f3}) holds. 
\end{proof}

\begin{rem} - In case $\goth g=\goth{osp}(2k,2k)$ Theorem ~\ref{thch} provides the formula 
for $Ch (L_{\lambda})$ only for positive $\lambda$. For negative
$\lambda$, we apply the automorphism $\sigma$ defined in Section 7, and we use
$$D^{\lambda,\mu}=D^{\lambda',\mu'}.$$
\end{rem}

{\bf Example.} Let $\goth g=\goth{osp}(6,6)$ and
$\lambda=(2,1,0|2,1,0)$. To find $Ch (L_{\lambda})$, we just have to
describe the subgraph $D_{\goth g}$ containing the vertices
$\mu\leq\lambda$. There are four such vertices corresponding to the
weights $\lambda$, $\mu=(2,0,0|2,0,0)$, $\nu=(1,0,0|1,0,0)$, $\kappa=0$.
They are connected by the edges:

$$\kappa\xrightarrow[]{(0,1;0)} \nu\,$$
$$\kappa\xrightarrow[]{(0,1;4)} \nu,$$ 
$$\kappa\xrightarrow[]{(0,2;3)} \mu,$$ 
$$\nu\xrightarrow[]{(1,2;0)} \mu,$$
$$\nu\xrightarrow[]{(0,2;1)} \lambda,$$ 
$$\nu\xrightarrow[]{(0,2;1)} \lambda,$$
$$\mu\xrightarrow[]{(0,1;0)} \lambda,$$
$$\mu\xrightarrow[]{(0,1;2)} \lambda.$$ 

Then the corresponding matrices are

$$\mathbb K=\left({\begin{matrix} 1 & 0 & 0 & 0 \\ 
2 & 1& 0 &0 \\ 0 &1 &1 &0 \\-2 & -1 & 2 &1\end{matrix}}\right)
$$

$$\mathbb D=\left({\begin{matrix} 1 & 0 & 0 & 0 \\ 
-2 & 1& 0 &0 \\ 2 &-1 &1 &0 \\-4 & 3 & -2 &1\end{matrix}}\right)
$$
and
$$Ch (L_\lambda)=\mathcal E_\lambda-2\mathcal E_\mu+2\mathcal E_\nu-4\mathcal E_\kappa.$$

{\bf The case $k=1$.} (See \cite{Jer}, \cite{G1}) If $\goth g=\goth {osp}(2,2)$, positive weights in the
most atypical block are of the form $a\varepsilon_1+a\delta_1$. To
simplify notations we put $L_{a}=L_{a\varepsilon_1+a\delta_1}$. The
graph $D_{\goth g}$ is the infinite string
$$0\xrightarrow[]{(0,1;0)} 1\xrightarrow[]{(1,2;0)} 2\xrightarrow[]{(2,3;0)}...,$$
and the characters can be calculated by
$$Ch(L_{a})=\sum_{j=0}^a(-1)^{a+j}\mathcal E_j.$$

The matrix $\mathbb D$ is the same for $\goth {osp}(4,2)$ and $\goth {osp}(3,2)$. Let us consider the latter case. 
Set $\lambda_0=0$, $\lambda_1=\varepsilon_1$,
$\lambda_i=i\varepsilon_1+(i-1)\delta_1$. The graph $D_{\goth g}$ has
the edges $\lambda_i\xrightarrow[]{(i,i+1;0)}\lambda_{i+1}$ for all
$i\geq 1$ and the edges $\lambda_0\xrightarrow[]{(0,2;0)}\lambda_{2}$,
$\lambda_0\xrightarrow[]{(0,1;1)}\lambda_{1}$.
One can easily obtain the character formulae
$$Ch(L_{\lambda_{p}})=\sum_{i=1}^p(-1)^{p+i}\mathcal E_{\lambda_i}-2(-1)^p \mathcal E_{\lambda_0}=\sum_{i=1}^p(-1)^{p+i}\mathcal E_{\lambda_i}-2(-1)^p ,$$
for $p\geq 2$, and
$$Ch(L_{\lambda_{1}})=\mathcal E_{\lambda_1}+\mathcal E_{\lambda_0}=\mathcal E_{\lambda_1}+1.$$

\section{Caps and cancellations}

Let $f$ be a weight diagram. For every $\times$ at  a non-tail position $s$,
there exists exactly one legal move $f\to f'$ of degree zero with
start at $s$. If $t$ is the end of that move, then we join $s$ and
$t$ by a cap. Proceeding in this way we equip $f$ with caps for each
non-tail $s$ such that $f(s)=\times$. A non-tail position $s$ is
called {\it free} if $f(s)=0$ and $s$ is not an end of a cap. 
One can easily check the following properties:

- there are no free positions under a cap;

- two caps do not overlap;

- if $s\neq s_0$, $f(s)=\times$, then the end of a legal move with start
at $s$ is not larger than the end of the cap starting at $s$.

We call a decreasing path $R\in \mathcal P^{>}(\mu,\lambda)$ {\it
  regular} if 

- any edge of $R$ corresponding to a non-tail legal move corresponds to a move 
along a cap on the diagram $f_{\mu}$;

- any edge of $R$ corresponding to a tail legal move or an exceptional move 
corresponds to a move with end at a free position.

Note that all ordinary and exceptional moves which appear in a regular path 
must have degree 0. In addition, it follows directly from Definition 
\ref{expmove} that the position $s$
in any exceptional legal move $(s_0:s,t;w)$ which appears in a decreasing path 
from $\mu$ to $\lambda$ is a free position of $f_{\mu}$.

\begin{prop}\label{canc} -  Let $\mathcal {RP}^{>}(\mu,\lambda)$ denote
  the set of all regular decreasing paths from $\mu$ to $\lambda$. Then 
$$\sum_{R\in \mathcal P^{>}(\mu,\lambda)}(-1)^{l(R)}=\sum_{R\in \mathcal 
{RP}^{>}(\mu,\lambda)}(-1)^{l(R)}.$$
\end{prop}

\begin{proof} -  Define an involution $*$ on the set of all non-regular paths
  in $\mathcal P^{>}(\mu,\lambda)$ as follows. If $R$ is a
  non-regular path, one can find at least one cap whose left or right
  end is the end of some ``wrong'' legal move included in $R$
which is not a move along this
  cap. Among such caps, pick up a cap with maximal left end. There are
  two possibilities:

1. The left end $t$ of this cap is the end of some ``wrong'' legal move
$(s,t;w)$ (resp.  ``wrong'' exceptional move $(s_0:s,t;w)$). 
Then, before the edge $(s,t;w)$ (resp. $(s_0:s,t;w)$) in the path $R$, 
there is an edge
$(t,u;0)$ which, by our conditions, corresponds to a legal move along a cap.
Exchange two edges $(s,t;w)$ (resp. $(s_0:s,t;w)$) and $(t,u;0)$ by one which
corresponds to a legal move $(s,u;w+1)$ (resp. exceptional move 
$(s_0:s,u;w+1)$) and get a new irregular path $R^*$.
One can easily check that $l(R^*)=l(R)+1$.

2. The right end $t$ of the cap is the end of some ``wrong'' legal move
$(s,t;w)$ (resp.  ``wrong'' exceptional move $(s_0:s,t;w)$). Note that in 
this case $w$ must be positive. 
Let $u$ be the left end of the cap. Remove the edge $(s,t;w)$ 
(resp. $(s_0:s,t;w)$) and insert $(u,t;0)$ and $(s,u;w-1)$ (resp. 
$(s_0:s,u;w-1)$) and 
get a new irregular path $R^*$.
One can easily check that $l(R^*)=l(R)-1$.

Obviously, $*$ is an involution. The statement follows from the condition
$(-1)^{l(R^*)}=-(-1)^{l(R)}$.
\end{proof}

\begin{prop}\label{mult1} - If $\goth g=\goth {osp}(2n+1,2n)$ or 
$\goth {osp}(2n+2,2n)$ and $\lambda,\mu$ have the trivial central
character, then $K^{\lambda,\mu}_{G,Q_\lambda}(-1)$ is either zero or
 $\pm 1$. In other words, the entries of
$\mathbb K$ are $0$ or $\pm 1$. 

If $\goth g=\goth {osp}(2n,2n)$  and $\lambda,\mu$ have trivial central
character, then $|K^{\lambda,\mu}_{G,Q_\lambda}(-1)|\leq 2$
\end{prop}
\begin{proof} - By Proposition \ref{canc} and Theorem \ref{combth} we
  have 
\begin{equation}\label{reg}
K^{\lambda,\mu}_{G,Q_\lambda}(-1)=\sum_{R\in \mathcal {RP}^{>}(\mu,\lambda)}
(-1)^{l(R)}.
\end{equation}

Let $t_1,t_2,...$ be all free positions written in increasing order. 
Let $R$ be some decreasing regular path. We call two tail moves appearing in 
$R$  {\it adjacent} if they have ends $t_i$ and $t_{i-1}$. 
A pair of adjacent tail moves is vanishing if there exists an 
exceptional move with label $(s_0:t_{i-1},t_i;0)$. (In particular, $i$ is odd 
in case $\goth g=\goth {osp}(2n+1,2n)$ or 
$\goth {osp}(2n+2,2n)$ and even in case $\goth g=\goth {osp}(2n,2n)$.)

If $R$ is a regular but not strongly regular decreasing path, we pick up the 
first exceptional move or the first vanishing adjacent pair which appears in 
it, depending on what occurs earlier. Denote by $R'$ the path obtained from 
$R$ by substituting the vanishing pair instead of the first exceptional move 
(or respectively the exceptional move instead of the first vanishing pair). 
Then clearly $R''=R$ and $R$ and $R'$ cancell in the summation of (\ref{reg}).
Hence the sum can be taken only over strongly regular paths.

In case $\goth g=\goth {osp}(2n+1,2n)$ or 
$\goth {osp}(2n+2,2n)$ one can see immediately that there is at most one 
strongly regular path between any two weight diagrams.

If $\goth g=\goth {osp}(2n,2n)$ there are two tail moves with the end $t_1$, 
one of degree 0 and one of degree equal to the double size of the tail. Hence 
there are at most two strongly regular paths between two weight diagrams.
Hence the statement.
\end{proof}

\section{Appendix: index of definitions and notations}

\begin{center}{\sc General setting}
\end{center}

\el

\noi - Integral dominant weight: Section 2.

\el

\noi - $\Phi$: just before Lemma ~\ref{l4}.

\el

\noi - $A(\lambda)$: Definition ~\ref{DAl}.

\el

\noi - Degree of atypicality of a dominant weight (resp. central character): Definition ~\ref{DAl} (resp. just after this definition).

\el

\noi - Admissible parabolic subalgebra for a central character: Definition ~\ref{Dpa}.

\el

\noi - Poincar\'e polynomial $K_{A,B} ^{\lambda , \mu} (z)$: Definition ~\ref{Dpoin}.

\el

\noi - List of simple roots for the chosen Borel subalgebras: Section 5.

\el

\noi - Dominance contitions for weights: Section 5.

\el

\noi - Core of a central character: just before Lemma ~\ref{Lcore}.

\el

\noi - $\goth g _\chi$ (for a central character $\chi$): before Lemma ~\ref{l3}.

\el

\noi - Stable weights: before Lemma ~\ref{l3}.

\el

\noi - $\mathcal F _{\leq \lambda} ^{\chi}$: before Lemma ~\ref{l3}.

\el

\noi - Translation functors $T(V)^{\chi , \tau}$: after the proof of Lemma ~\ref{l3}.

\el

\noi - $E$: Standard $\goth g$-module.

\el

\noi - Case $\goth{osp}(2m,2n)$, positive (resp. negative) weights: Section 6.

\el

\noi - Tail subalgebra of a dominant weight $\lambda$, $\goth g _\lambda$; tail of $\lambda$; algebra $\goth q _{\lambda}$: just before Lemma ~\ref{reduction1}.

\el

\noi - Switch functor $T(E)^{\chi , \chi}$: after the proof of Lemma ~\ref{funsw}.

\el

\noi - Exceptional pair of weights: before Proposition ~\ref{pt2}.

\el

\noi - Pretail weights: Section 9.

\el
\begin{center} {\sc Weight diagrams and algorithm}
\end{center}

\el

\noi - Diagram $f_{\lambda}$ associated to a weight $\lambda$: defined in the beginning of Section 6. We will use the notations introduced there for this part of the appendix.

\el

\noi - Diagram of the core: the diagram one obtains from $f_{\lambda}$ when removing all the symbols $\times$.

\el

\noi - Is it possible to recover the weight from the diagram? Yes if $\goth g = \goth {gl}(m,n)$, no if $\goth g = \goth {osp}(m,2n)$ (see Section 6).

\el

\noi - Number of $\times$-s in the diagram: atypicality degree of the corresponding weight.

\el

\noi - Indicator $\pm$ and sign: see $\goth{osp}(2n+1,2n)$ in Section 6 (depends on the tail's shape).

\el

\noi - Action of translation functor: see section 6.

\el

\noi - Relationship between the number of $\times$-s at $0$ and the length of the tail of $\lambda$: just before Lemma \ref{reduction1}.

\el

\noi - Core symbols: beginning of Section 10.

\el

\noi - Notation $s_0$: tail position.

\el

\noi - Notation $l_f(s,t)$: number of $\times$-s $-$  number of $0$-s strictly between the positions $s$ and $t$. 

\el

\noi - Notation $\vert f \vert$: $2\times$ number of $\times$-s at the tail position + number of core symbols at the tail position.

\el

\noi - Notation $f_t$: If $f$ has a $\times$ at position $t$ , the $f_t$ is the diagram obtained when removing it.

\el

\noi - Notation $f^t$: Diagram obtained adding to $f$ a $\times$ at position $t$.

\el

\noi - Notation $f_s ^t$: diagram obtained from $f$ moving a $\times$ from $s$ to $t$.

\el

\noi - Legal moves, start and end of a legal move, degree of a legal move: Definition 5, Section 11.

\el

\noi - Tail legal move, ordinary legal move: Definition 5, Section 11.

\el

\noi - Exceptional move: Definition 6, Section 11.

\el

\noi - Decreasing, increasing paths in the oriented graph $D_{\goth g}$: just after Corollary ~\ref{legmovcor2}.

\el

\noi - Length of a path in a $D_{\goth g}$: just before Theorem \ref{combth}.
\el

\noi - Cap: Section 13.

\el

\noi - Free position: Section 13.

\el

\noi - Strongly regular: Section 13.

\el

\noi - Regular path: Section 13.

\el


\begin{thebibliography}{10}

\bibitem{BG} J. Bernstein, S. Gelfand,  Tensor products of finite and
  infinite-dimensional representations
of semisimple Lie algebras, {\it Compositio Math.} {\bf 41}
(1980), 245--285.

\bibitem{BL} J. Bernstein, D. Leites,
A formula for the characters of the irreducible finite-dimensional representations of Lie superalgebras of series ${Gl}$ and ${sl}$. (Russian)  C. R. Acad. Bulgare Sci.  33  (1980), no. 8, 1049--1051.

\bibitem{B} J. Brundan,
Kazhdan-Lusztig polynomials and character formulae for the Lie superalgebra $gl(m\vert n)$.  J. Amer. Math. Soc.  16  (2003),  no. 1, 185--231.


\bibitem{BS1} J. Brundan, C. Stroppel, Highest weight categories 
arising from Khovanov's diagram algebra I: cellularity.  Preprint, 2008.

\bibitem{BS2} J. Brundan, C. Stroppel, Highest weight categories 
arising from Khovanov's diagram algebra II: Koszulityty.  Preprint, 2008.

\bibitem{BS3} J. Brundan, C. Stroppel, Highest weight categories 
arising from Khovanov's diagram algebra III: Category $\mathcal{O}$.  
Preprint, 2008.

\bibitem{CL} S. J. Cheng, N. Lam, Irreducible characters of general linear superalgebra and super duality. arXiv:0905.0332v1.

\bibitem{CW} S. J. Cheng, W. Wang, Brundan-Kazhdan-Lusztig and
  superduality conjecture. Publ. RIMS, Kyoto Univ. 44 (2008), 1219--1272.



\bibitem{DS} M. Duflo, V. Serganova, 
On associated variety for Lie superalgebras, math/0507198.


\bibitem{Jer} J. Germoni,
Indecomposable representations of ${\rm osp}(3,2)$, $D(2,1;\alpha)$ and $G(3)$. Colloquium on Homology and Representation Theory (Spanish) .  Bol. Acad. Nac. Cienc. (Cordoba)  65  (2000), 147--163. 

 \bibitem{Gor} M. Gorelik,
Annihilation theorem and separation theorem for basic classical Lie superalgebras.  J. Amer. Math. Soc.  15  (2002),  no. 1, 113--165 
 
\bibitem{G2} C. Gruson,
Sur l'ideal du cone autocommutant des super algebres de Lie basiques classiques et etranges. (French) [On the ideal of the self-commuting cone of basic classical and strange Lie superalgebras],  Ann. Inst. Fourier (Grenoble),  50  (2000),  no. 3. 807--831.


\bibitem{G1} C. Gruson,
Cohomologie des modules de dimension finie sur la super algebre de Lie ${osp}(3,2)$. (French) [Cohomology of finite-dimensional modules over the Lie superalgebra ${osp}(3,2)$]  J. Algebra  259  (2003),  no. 2, 581--598.

\bibitem{G-L} C. Gruson, S. Leidwanger,
Cones nilpotents des super algebres de Lie orthosymplectiques
(preprint available on Severine Leidwanger's web page)
  
\bibitem{J} J. Jantzen,
Representations of algebraic groups. Second edition. Mathematical Surveys and Monographs, 107. American Mathematical Society, Providence, RI, 2003.


\bibitem{Kadv} V. Kac,
Lie superalgebras.  Advances in Math.  26  (1977), no. 1, 8--96. 

\bibitem{Krep} V. Kac,
Characters of typical representations of classical Lie superalgebras.  Comm. Algebra  5  (1977), no. 8, 889--897. 

\bibitem{KW} V. Kac, M. Wakimoto, 
Integrable highest weight modules over affine
superalgebras and number theory.  Lie theory and geometry, 415--456, 
Progr. Math., 123, Birkhauser, Boston, MA, 1994. 


\bibitem{M} Yu. Manin,
Gauge field theory and complex geometry. Translated from the 1984 Russian original by N. Koblitz and J. R. King. Grundlehren der Mathematischen Wissenschaften [Fundamental Principles of Mathematical Sciences], 289. Springer-Verlag, Berlin, 1997.

\bibitem{MPV} Yu. Manin, I. Penkov, A. Voronov,
Elements of supergeometry. (Russian) Translated in J. Soviet Math. 51 (1990), no. 1, 2069--2083.

\bibitem{MV} Yu. Manin, A. Voronov,
Supercellular partitions of flag superspaces. (Russian) 
Translated in J. Soviet Math. 51 (1990), no. 1, 2083--2108.


\bibitem{MS} I. Musson, V. Serganova
Combinatorics of character formulas for the Lie superalgebra 
$\mathfrak{gl}(m,n)$, Preprint.


\bibitem{P} I. Penkov,
Borel-Weil-Bott theory for classical Lie supergroups. (Russian) Translated in J. Soviet Math. 51 (1990), no. 1, 2108--2140.


\bibitem{V.Sel} V. Serganova,
Kazhdan-Lusztig polynomials and character formula for the Lie superalgebra ${gl}(m\vert n)$.  Selecta Math. (N.S.)  2  (1996),  no. 4, 607--651.


\bibitem{V.ICM} V. Serganova,
Characters of irreducible representations of simple Lie superalgebras. Proceedings of the International Congress of Mathematicians, Vol. II (Berlin, 1998).


\bibitem{V.adv} V. Serganova,
A reduction method for atypical representations of classical Lie superalgebras.  Adv. Math.  180  (2003),  no. 1, 248--274.


\bibitem{V.KM} V. Serganova,
Kac-Moody superalgebras and integrability. To appear in ``Perspectives
of infinite-dimensional Lie theorey'', Birkhauser.

\bibitem{S} A. Sergeev,
The invariant polynomials on simple Lie superalgebras.  Represent. Theory  3  (1999), 250--280. 

\bibitem{VdJ} J. Van der Jeugt, Character formulae for Lie
  superalgebra $C(n)$, Comm. Algebra 19, (1991), no. 1, 199--222. 

\bibitem{JHKT} J. Van der Jeugt, J. W. B. Hughes, R. C. King, J. Thierry-Mieg,
A character formula for singly atypical modules of Lie superalgebra
$sl(m,n)$, Comm. Algebra 18, (1990), no. 10, 3453--3480.

\end{thebibliography}
\end{document}